%% file: Weighted_Hodge_ideals.tex
\documentclass[11pt]{amsart}
\usepackage[utf8]{inputenc}
\usepackage{kantlipsum} % for text filler
\calclayout
\newif\ifdraft
\draftfalse
\input{macros.tex}

\DeclareMathOperator{\adj}{adj}
\DeclareMathOperator{\Gr}{Gr}
\newcommand{\remarkref}[1]{\hyperref[#1]{Remark~\ref*{#1}}}

\begin{document}

\vspace{\baselineskip}

\title{Weighted Hodge ideals of reduced divisors}

\author[Sebasti\'an~Olano]{Sebasti\'an~Olano}
\address{Department of Mathematics, University of Toronto, 40 St. George St., Toronto, Ontario, Canada , M5S 2E4} \email{{\tt seolano@math.toronto.edu}}

\thanks{}

%\subjclass[2010]{14J17, 14F17, 32S25, 32S35}

\begin{abstract} We study the Hodge and weight filtrations on the localization along a hypersurface, using methods from birational geometry and the $V$-filtration induced by a local defining equation. These filtrations give rise to ideal sheaves called weighted Hodge ideals, which include the adjoint ideal and a multiplier ideal. We analyze their local and global properties, from which we deduce applications related to singularities of hypersurfaces of smooth varieties.
\end{abstract}

\maketitle

\section{Introduction}

In this paper, we continue the study of weighted Hodge ideals that started in \cite{olano22a}, where the focus was the $0$-th weighted Hodge ideals, also called weighted multiplier ideals. We show that several results satisfied by the weighted multiplier ideals can be generalized under suitable conditions.\\

Let $X$ be a smooth complex variety of dimension $n$. To an effective reduced divisor $D$ on $X$ one can associate a sequence of ideal sheaves $I_p(D)\subseteq \shO_X$, called the Hodge ideals of $D$ and studied in a series of papers \cite{hodgeideals},\cite{mustatapopa18},\cite{mustatapopa19}, \cite{mustatapopa20b}, \cite{mustatapopa20}. They arise from the theory of mixed Hodge modules of M. Saito, which induces a Hodge filtration $F_{\bullet}\OX(*D)$ by coherent $\OX$-modules on $\OX(*D)$, the sheaf of functions with poles along $D$, seen as a left $\Dmod_X$-module. This $\Dmod$-module underlies the mixed Hodge module $j_*\QQ^H_{U}[n]$, where $j: U=X\setminus D\into X$. Saito showed that the Hodge filtration is contained in the pole order filtration, that is,  $$F_p\OX(*D) \subseteq \OX((p+1)D)$$ for all $p\geq 0$. Consequently, we can define the Hodge ideal $I_p(D)$ by $$F_p\OX(*D) = \OX((p+1)D) \otimes I_p(D).$$ The $\Dmod_X$-module $\OX(*D)$ is also endowed with a weight filtration $W_{\bullet}\OX(*D)$ by $\Dmod_X$-submodules. The Hodge filtration of these submodules satisfies $$F_pW_{n+l}\OX(*D) \subseteq F_p\OX(*D) \subseteq \OX((p+1)D),$$ and similarly we can define the weighted Hodge ideals by $$F_pW_{n+l}\OX(*D) = \OX((p+1)D) \otimes I^{W_l}_p(D).$$

The weighted Hodge ideals form a chain of inclusions $$I_p^{W_0}(D) \subseteq I_p^{W_1}(D) \subseteq \cdots \subseteq I_p^{W_n}(D).$$ We can always understand the two extreme ideals in this chain. The first element in the list admits an easy description: $$I_p^{W_0}(D) = \OX(-(p+1)D).$$ On the other end, the last ideal in this chain is the usual $p$-th Hodge ideal, that is, $$I_p(D) = I_p^{W_n}(D).$$ Unlike $I_p^{W_0}(D)$, for all the other degrees, the support of the scheme defined by $I_p^{W_l}(D)$ is contained in the singular locus of $D$.\\

\noindent\textbf{Birational definition} We give an alternative description of the weighted Hodge ideals in terms of a resolution of singularities. Let $f:Y\to X$ be a resolution of singularities of the pair $(X,D)$ which is an isomorphism over $X\setminus D$, and let $E := (f^*D)_{\rm red}$. This description stems from the birational definition of Hodge ideals in \cite{hodgeideals}*{\textsection{9}}, and uses right $\Dmod$-modules. The $\Dmod_Y$-module $\omega_Y(*E)$ admits a filtered resolution by $\Dmod_Y$-modules given by $$B^{\bullet} = 0\to \Dmod_Y\to \Omega^1_Y(\log{E})\otimes_{\shO_Y} \Dmod_Y\to \cdots\to \omega_Y(E)\otimes_{\shO_Y}\Dmod_Y\to 0.$$ Similarly, using the weight filtration on the sheaves of logarithmic $p$-forms (see (\ref{weightlogarithmic})), we show that the complex $$W_lB^{\bullet}=0\to \Dmod_Y\to W_l\Omega^1_Y(\log{E})\otimes_{\shO_Y} \Dmod_Y\to \cdots\to W_l\omega_Y(E)\otimes_{\shO_Y}\Dmod_Y\to 0$$ is filtered quasi-isomorphic to the $\Dmod_X$-module $W_{n+l}\omega(*E)$ (see \propositionref{resolutionmodule}).\\

The $\Dmod_X$-module $\omega_X(*D)$ can be described using the filtered resolution of $\omega_Y(*E)$ described above. More precisely, we can define the complex $A^{\bullet}$ by $$0\to f^*\Dmod_X\to \Omega^1_Y(\log{E})\otimes_{\shO_Y} f^*\Dmod_X\to \cdots\to \omega_Y(E)\otimes_{\shO_Y}f^*\Dmod_X\to 0$$ placed in degrees $-n, \ldots, 0$, and we have that, $$R^0f_*A^{\bullet} \cong \omega_X(*D)$$ (see \cite{hodgeideals}*{\textsection{9}}). To give the alternative description of the weighted Hodge ideals, we introduce the complex $C^{\bullet}_{l,p-n}$ defined as $$ 0\to f^*F_{p-n}\Dmod_X\to W_l\Omega^1_Y(\log{E})\otimes_{\shO_Y} f^*F_{p-n+1}\Dmod_X\to \cdots\to W_l\omega_Y(E)\otimes_{\shO_Y}f^*F_p\Dmod_X\to 0$$ and we show that the image of $$R^0f_*C^{\bullet}_{l,p-n} \to R^0f_*A^{\bullet}= \omega_X(*D)$$ is precisely $F_{p-n}W_{n+l}\omega_X(*D) = I^{W_l}_p(D) \otimes \omega_X((p+1)D)$ (see \propositionref{birationaldef}).\\

\noindent\textbf{Description of weighted Hodge ideals using the $V$-filtration.} A very convenient local description of Hodge ideals was given in terms of the Kashiwara-Malgrange $V$-filtration of the graph embedding $i_+\OX$ in \cite{mustatapopa20b}*{Theorem A'} (see (\ref{hivfiltration})), which works in the more general setting of Hodge ideals of $\QQ$-divisors. In this case, we suppose that the reduced divisor $D\subseteq X$ can be defined by a regular function $f\in\OX(X)$. Weighted Hodge ideals admit a similar description.

\begin{intro-theorem}\label{whivfiltration2} Let $X$ be a smooth complex variety and $D$ a reduced divisor defined by a regular function $f\in \OX(X)$. Then, $$I^{W_l}_p(D) = \left\{ \sum_{j=0}^p Q_j(1)f^{p-j}v_j \ : v =\sum_{j=0}^p v_j\partial_t^j\delta\in V^{1}i_+\OX \text{ and } (t\partial_t)^lv \in V^{>1}i_+\OX\right\}.$$ 
\end{intro-theorem}

The proof is based on two ideas. First, we can relate the Hodge filtration of $V^1i_+\OX$ with that of $\OX(*D)$ (see \ref{hodgevfiltration}). Second, the weight filtration on the nearby cycles sheaf can be related to that of the local cohomology sheaf (\propositionref{grvtau}). This is enough to understand all the weighted Hodge ideals in the case when $D$ only has isolated weighted-homogeneous singularities (see \remarkref{qhomog}).\\

The description in \theoremref{whivfiltration2} is useful to relate the weighted Hodge ideals with some invariants of the singularities, like the minimal exponent. Recall that to the variety $D\subseteq X$ we can associate the Bernstein-Sato polynomial $b_D(s)$. The polynomial $(s+1)$ divides $b_D(s)$, and we denote $\widetilde{b_D}(s) = b_D(s)/(s+1)$. The negative of the largest root of $\widetilde{b_D}(s)$ is called the \textit{minimal exponent} of a $D$ and is denoted $\widetilde{\alpha_D}$. This invariant encodes important properties of the singularities of $D$. For instance, it is a refined version of the log-canonical threshold, since $lct(X,D) = \min\{\widetilde{\alpha_D}, 1\}$. In particular, this implies that $(X,D)$ is log-canonical if and only if $\widetilde{\alpha_D} \geq 1$. Moreover, it is a result of Saito that $D$ has rational singularities if and only if $\widetilde{\alpha_D}>1$.\\

The notions of log-canonicity and rationality can be described in terms of weighted Hodge ideals. Recall that 0-th weighted Hodge ideals, or weighted multiplier ideals, form a sequence of ideals interpolating between the adjoint ideal and a multiplier ideal. This is the case, as $I_0^{W_1}(D) = \adj(D)$ (see for instance \cite{olano22a}*{Theorem A}) and $I_0(D) = \mathcal{J}((1-\varepsilon)D)$  for $0<\epsilon\ll 1$ \cite{budursaito05}. These two ideals identify if a singularity is respectively rational or log-canonical. We give an analogous description for the higher weighted Hodge ideals. The Hodge ideal $I_p(D)$ is trivial if and only if $\widetilde{\alpha_D}\geq p+1$, in which case we say that $(X,D)$ is $p$-log-canonical. Also, the weighted Hodge ideal $I_p^{W_1}(D)$ is trivial if and only if $\widetilde{\alpha_D}> p+1$ (see \corollaryref{whiminimalexp}), which some authors referred to as $D$ being $p$-rational. The rest of the $p$-weighted Hodge ideals filter and measure the ``distance'' between $(X,D)$ having $p$-log-canonical singularities and $D$ being $p$-rational.\\

\noindent\textbf{Isolated singularities.} Recall that the weighted Hodge ideals satisfy $$I_p^{W_{l-1}}(D)\subseteq I_p^{W_l}(D).$$ The difference between the two ideals can be described by the coherent sheaf $F_p\gr^W_{n+l}\OX(*D)$ (see (\ref{differencesequence})). If $D$ has isolated singularities, we give a description of the dimension of this sheaf at the singular points in terms of a resolution of singularities. For this, possibly after restricting to an open set, assume $D$ has one isolated singularity $x\in D$. In this case, there exists a pure Hodge structure $H_l$ for $l\geq 2$, such that the dimension of their Hodge pieces describes the desired dimension. More concretely,  \begin{equation}\label{equationdifference} \dim(F_p(\gr^W_{n+l}\OX(*D))_x) = \sum_{r=0}^p\binom{n+p-r}{p-r}\dim(\Gr_F^{n-r}H_l)\end{equation} (see \textsection \ref{sectiondifference} for more details). For this reason, to find the difference between two consecutive weighted Hodge ideals, it is enough to compute the dimensions of the spaces $\Gr_F^{n-p}H_l$. 

\begin{intro-theorem}\label{thmdifference} Let $g:\widetilde{D}\to D$ be a log-resolution of singularities that is an isomorphism outside of $x$. Let $G\subseteq\widetilde{D}$ be the exceptional divisor.  Then $$\dim(\Gr_F^{n-p}H_l) = h^{p,n-l-p}(H^{n-2}(G))$$ if $l\geq 3$, and $$\dim(\Gr_F^{n-p}H_2) = h^{p,n-p-2}(H^{n-2}(G)) - h^{n-p-1,p+1}(H^n(G)),$$ where $H^k(G) = H^k(G,\CC)$ and $h^{p,q}(H^k(G)) = \dim(H^{p,q}(\Gr^W_{p+q}H^k(G)))$.
\end{intro-theorem}

When $p =0$ the second summand in the description of $\dim(\Gr_F^{n-p}H_2)$ is 0 because the dimension of $G$ is $n-2$, and therefore these dimensions are described as Hodge numbers of the middle cohomology of $G$. For $p\geq 1$ we cannot expect this term to be 0 in general, but this dimension admits a geometric interpretation (see \remarkref{rmksecondterm}).\\

\noindent\textbf{Vanishing results.} Weighted Hodge ideals satisfy global results under suitable conditions. Let $X$ be a smooth projective variety and $D$ an ample divisor with at most isolated singularities. Under this assumptions, when $p=0$ we have that $$H^i(X, \omega_X(D)\otimes I_0^{W_l}(D)) = 0$$ for $i\geq 1$ and $l\geq 2$ \cite{olano22a}*{Theorem E}. To generalize this result for all $p\geq 1$, we require the condition that $I_{p-1}^{W_l}(D) = \OX$. 

\begin{intro-theorem}\label{vanishingample} Let $X$ be a smooth projective variety of dimension $n$, and $D$ an ample reduced effective divisor with at most isolated singularities. Suppose that $I_{p-1}^{W_1}(D)$ is trivial. Then
\begin{enumerate}\item For $l\geq 2$ and $i\geq 2$,  $$H^i(X, \omega_X((p+1)D) \otimes I_p^{W_l}(D))= 0.$$ 
\item If $H^j(X, \Omega_X^{n-j}((p-j+1)D)) = 0$ for all $1\leq j\leq p$, then $$H^1(X, \omega_X((p+1)D) \otimes I_p^{W_l}(D))= 0$$ for $l\geq 2$.
\end{enumerate}
\end{intro-theorem}

When $l=1$ and $i=1$ the vanishing does not hold in general. For an example see \remarkref{rmkvanishingample}. A Kodaira-type vanishing result is also satisfied for all $l\geq 1$, and the proof is based on a vanishing result by Saito \cite{saito90}*{Proposition 2.33} (see \propositionref{vanishingkodaira}).\\

\noindent\textbf{Applications.} The global and local results we have discussed can be used to obtain results about the geometry of certain isolated singularities of hypersurfaces in $\PP^n$. This is because the vanishing condition in \theoremref{vanishingample} is satisfied when $X$ is a toric variety.

\begin{intro-corollary}\label{indconditions} Let $D\subseteq \PP^n$ be a hypersurface of degree $d$ with at most isolated singularities. Let $Z_{l,p}$ be the scheme defined by $I_p^{W_l}(D)$. Then, $$H^0(\PP^n, \shO_{\PP^n}(k)) \onto H^0(\PP^n, \shO_{Z_{l,p}})$$ for $k\geq (p+1)d-n-1$ if $l\geq 2$, and $k\geq (p+1)d - n$ if $l\geq 1$.
\end{intro-corollary}

This result gives a bound on a certain type of isolated singularities we describe next. For simplicity, suppose $D$ has at most one isolated singularity $x\in D$, and assume $\widetilde{\alpha_D} = p+1$. We describe first the case $p=0$. This case corresponds to a log-canonical and not rational singularity. In this case, according to (\ref{equationdifference}), the length of the scheme described by $I_0^{W_1}(D)$ is determined by $\Gr_F^0(H^{n-2}(G))$, using the notation of \theoremref{thmdifference}. Ishii proved that in this case, $\dim(\Gr_F^0(H^{n-2}(G))) = 1$ \cite{ishii85}*{Proposition 3.7}. This means that the ideal $I_0^{W_1}(D)$ is the maximal ideal of $x$ in $X$, and that there exists exactly one degree $l\geq 2$ such that $$\dim(\Gr_F^n(H_l)) = 1,$$ while the dimension for the other degrees is 0. A log-canonical singularity is of type $(0,n-l)$ in this case \cite{ishii85}*{Definition 4.1}.\\

Assume now that $\widetilde{\alpha_D} = p+1$ for an $p\in\ZZ_{\geq 0}$, and that $D$ has at most one isolated singularity $x\in D$. In this case, we have an analogous picture. Namely, the ideal $I_p^{W_1}(D)$ is the maximal ideal of $x$ in $X$ (see \propositionref{minexprk1}), or equivalently, as $I_p(D) = \OX$, the length of the scheme described by $I_p^{W_1}(D)$ is 1. This in particular means that $$\Gr_F^{n-r}H_l = 0$$ for $l \geq 2$ and $0\leq r \leq p-1$ by (\ref{equationdifference}) and \theoremref{thmdifference}. Moreover, by the same results, we know that there exists exactly one degree $l\geq 2$ such that $$\dim(\Gr_F^{n-p}H_l) = 1,$$ while the dimension for all the other degrees is 0. Related invariants in similar conditions have been studied by Friedman and Laza in \cite{friedmanlaza22b}*{Theorem 6.11 and Corollary 6.14}.\\

In analogy to the case of log-canonical singularities, we call the singularity described above of type $(p,n-l-p)$ (see \definitionref{(p,s-p)}). Weighted homogeneous singularities with $\widetilde{\alpha_f} = p+1$ are examples of singularities of type $(p,n-2-p)$ and the origin in $Z(x^2+y^2+z^2+u^2w^2+u^4+w^5)\subseteq \mathbb{A}^5$ gives an example of a singularity of type $(1,5-3-1) = (1,1)$ (see \exampleref{examplessings}). For a hypersurface of $\PP^n$ with at most isolated singularities and $\widetilde{\alpha_D} = p+1$, we give a bound on the number of these singularities (see \corollaryref{countpoints}). \\

\noindent\textbf{Restriction theorem.} Finally, we study the behavior of weighted Hodge ideals of a pair $(X,D)$ under the restriction of a hypersurface of $X$. Let $H\subseteq X$ be a smooth hypersurface, and $D_H$ the restriction of $D$ to $H$. If $D_H$ is reduced, then we can also consider the pair $(H, D_H)$ and their respective weighted Hodge ideals.

\begin{intro-theorem}\label{thmrestriction} Let $X$ be a smooth variety and $D$ an effective reduced divisor. Let $H\subseteq X$ be a smooth divisor such that $H\subsetneq \Supp(D)$ and $D_H = D\restr{H}$ is reduced. Then, for every $p\geq 0$ and $l\geq 0$ we have $$I_p^{W_l}(D_H)\subseteq I_p^{W_l}(D)\cdot\shO_H.$$ Moreover, if $H$ is general, then we have an equality. \end{intro-theorem}

This is the analogue of the Restriction Theorem for Hodge ideals \cite{mustatapopa18}*{Theorem A}, and for multiplier ideals \cite{lazarsfeld2}*{Theorem 9.5.1}. 

\medskip

\noindent
{\bf Acknowledgements.}
I would like to thank Mircea Musta\c{t}\u{a} and Mihnea Popa for their constant support and many conversations during the project. I am also very grateful to the anonymous referees for their feedback on improving the presentation of the article and for suggesting a simpler proof of \lemmaref{lemmaautomorphism}, explained in \remarkref{remarkreferee}.

\section{Preliminaries} 

\subsection{Mixed Hodge modules}\label{mhmpreliminaries} In this section, we recall some facts about mixed Hodge modules and set up the notation we use throughout this paper.\\

Let $X$ be a smooth variety of dimension $n$. Mixed Hodge modules introduced by Saito in \cite{saito88} are the main object used throughout this article. For a graded-polarizable mixed Hodge module $M$, we denote the underlying left regular holonomic $\Dmod_X$-module by $\Mmod$. In some contexts, it is more useful to use right $\Dmod_X$-modules. Recall that if $\Mmod$ is a left $\Dmod_X$-module, the corresponding right $\Dmod_X$-module is $\Mmod\otimes_{\OX}\omega_X$, where $\omega_X$ is the canonical sheaf. We mostly use left $\Dmod$-modules, and in case we are using right $\Dmod$-modules instead, we will say it explicitly.\\

A mixed Hodge module $M$ is endowed with a weight filtration, which we denote by $W_{\bullet}M$, and $$\gr^W_lM := W_lM/W_{l-1}M$$ is the quotient, which is a polarizable Hodge module of weight $l$. We denote by $F_{\bullet}\Mmod$ the Hodge filtration. The de Rham complex is defined as: $$\DR(\Mmod) = \big[\Mmod\to\Omega^1_X\otimes_{\shO_X}\Mmod\to\cdots\to\omega_X\otimes_{\shO_X}\Mmod\big][n],$$ and the Hodge filtration of $\Mmod$ induces a filtration on this complex: $$F_p\DR(\Mmod) = \big[F_p\Mmod\to\Omega^1_X\otimes_{\shO_X}F_{p+1}\Mmod\to\cdots\to\omega_X\otimes_{\shO_X}F_{p+n}\Mmod\big][n].$$ The $p$-th subquotient of this filtration is the complex $$\gr^F_p\DR(\Mmod) = \big[\gr^F_p\Mmod\to\Omega^1_X\otimes_{\shO_X}\gr^F_{p+1}\Mmod\to\cdots\to\omega_X\otimes_{\shO_X}\gr^F_{p+n}\Mmod\big][n].$$ 

Let $D$ be a reduced effective divisor. The mixed Hodge module we mostly study in this paper is $j_*\QQ^H_{U}[n]$, where $j:U=X\smallsetminus D\into X$, whose underlying $\Dmod_X$-module is the sheaf of functions with poles along $D$ denoted by $\shO_X(*D)$. To study $\shO_X(*D)$, it is sometimes convenient to use a resolution of singularities, and the properties of pushforwards. Fix a log-resolution of singularities of $(X,D)$, that is, a proper birational morphism $f:Y\to X$ such that $Y$ is smooth, it is an isomorphism over $U$, and $(f^*D)_{red} = E$ is a divisor with simple normal crossings. In this setup, we have that \begin{equation}\label{pushforward}f_+\OY(*E) \cong H^0f_+\OY(*E) \cong \OX(*D)\end{equation} (see for example \cite{hodgeideals}*{Lemma 2.2}). Since $E$ is a simple normal crossings divisor, the weight filtration of the $\Dmod_Y$-module $\OY(*E)$ can be described in terms of the intersections of its irreducible components. The lowest degree of the weight filtration is $n= \dim{Y}$, that is: $$W_{n-1}\OY(*E) = 0.$$ The lowest piece corresponds to the canonical Hodge module of $Y$: $$W_n\OY(*E) \cong \OY.$$ To describe the rest of the subquotients, we introduce the following very useful notation. Let $$E = \bigcup_{i\in I}E_i.$$ The variety $$E(l) = \bigsqcup_{\substack {J\subseteq I\\|J|=l}}E_J,$$ with $E_J = \displaystyle\bigcap_{j\in J} E_j$, is a smooth and possibly disconnected variety. We denote $i_l: E(l) \to Y$ the map such that on each component is the inclusion. We have that \begin{equation}\label{isoks}\gr^W_{n+l}\OY(*E) \cong i_{l+}\shO_{E(l)}\end{equation} with a Tate twist (see \cite{ks21}*{Prop 9.2}). \\

In order to describe the weight filtration of a pushforward of a projective morphism, a useful tool is to use the spectral sequence associated to the weight filtration: \begin{equation}\label{ss1}E_1^{p,q}= H^{p+q}f_+(\gr^W_{-p}\OY(*E))\Rightarrow H^{p+q}f_+\OY(*E),\end{equation}  
which degenerates at $E_2$, and there is an isomorphism: $$E_2^{p,q}\cong \gr^W_{q}H^{p+q}f_+\OY(*E)$$ \cite{saito90}*{Proposition 2.15}.\\

Finally, recall that the sheaf of  $p$-forms with logarithmic poles along $E$ denoted by $\Omega^p_Y(\log{E})$ are endowed with a weight filtration. This increasing filtration consists of subsheaves \begin{equation}\label{weightlogarithmic} W_l\Omega_Y^p(\log{E}) \subseteq \Omega_Y^p(\log{E})\end{equation} such that if $z_1, \ldots, z_n$ are local coordinates on an open set $V$, and $E$ is given by the equation $$z_1\cdots z_r=0,$$ then in $V$, $W_l\Omega^p(\log{E})$ is a $\shf{O}_V$ module generated by elements of the form $$\frac{dz_{i_1}}{z_{i_1}}\wedge \cdots \wedge \frac{dz_{i_s}}{z_{i_s}}\wedge dz_{j_1} \wedge \cdots \wedge dz_{j_{p-s}}$$ with $i_l\leq r$ and $s\leq k$ (see \cite{elzeinetal}*{3.4.1.2} for more details). For $I=\{i_1,\ldots , i_s\}$ and $J=\{j_1,\ldots, j_{p-s}\}$ we use the notation $$\frac{dz_I}{z_I}\wedge dz_J = \frac{dz_{i_1}}{z_{i_1}}\wedge \cdots \wedge \frac{dz_{i_s}}{z_{i_s}}\wedge dz_{j_1} \wedge \cdots \wedge dz_{j_{p-s}}.$$

\section{Characterizations}
\subsection{Definition} In this section, we introduce weighted Hodge ideals using the theory of mixed Hodge modules. \\

A fundamental result by Saito about the Hodge filtration on $\OX(*D)$ states that  $$F_p\OX(*D) \subseteq \OX((p+1)D)$$ (see \cite{saito93}*{Proposition 0.9}). The definition of Hodge ideals follows from this result. These ideals are denoted by $I_p(D)$, and are defined using the formula $$F_p\OX(*D) = I_p(D) \otimes\OX((p+1)D)$$ (see \cite{hodgeideals}*{Definition 9.4}). In this article, we study weighted Hodge ideals which are defined similarly using the weight filtration with which $\OX(*D)$ is endowed. The Hodge filtration of the sub-$\Dmod_X$ modules $W_{n+l}\OX(*D)$ satisfies $$F_pW_{n+l}\OX(*D) \subseteq F_p\OX(*D) \subseteq \OX((p+1)D)$$ for all $p\geq 0$.

\begin{definition}[Weighted Hodge ideals] Let $X$ be a smooth complex variety and $D$ a reduced divisor. For $l\geq 0$ and $p\geq 0$, we define the ideal sheaf $I_p^{W_l}(D)$ on $X$ by the formula $$F_pW_{n+l}\OX(*D) = I_p^{W_l}(D)\otimes\OX((p+1)D).$$ We call $I_p^{W_l}(D)$ the \textit{$l$-th weighted $p$-th Hodge ideal} of $D$. \end{definition} 

There is in fact a chain of inclusions \begin{equation}\label{chain} I_p^{W_{1}}(D)\subseteq I_p^{W_{2}}(D)\subseteq \cdots\subseteq I_p^{W_{n-1}}(D)\subseteq I_p^{W_{n}}(D)\end{equation} for all $p\geq 0$. Indeed, the weight filtration of $\OX(*D)$ is an increasing filtration, hence $$F_pW_{n+l}\OX(*D) \subseteq F_pW_{n+l+1}\OX(*D),$$ or equivalently $$\OX((p+1)D)\otimes I_p^{W_l}(D) \subseteq \OX((p+1)D)\otimes I_p^{W_{l+1}}(D).$$

\subsection{Simple normal crossings divisor}

Weighted Hodge ideals can be described completely when the reduced divisor $D$ has simple normal crossings. In this case, the Hodge filtration of $\OX(*D)$ is fully understood, and from this information we can deduce the Hodge filtration of $W_{n+l}\OX(*D)$.\\

Let $D$ be a simple normal crossings divisor. In this case, the Hodge filtration of $\OX(*D)$ admits a simple description: \begin{equation}\label{sncgen0} F_p\OX(*D) = F_p\Dmod_X\cdot\OX(D)\end{equation} if $p\geq 0$, and 0 otherwise. Using this, one obtains a local description of the Hodge ideals. Let $x_1, \ldots , x_n$ be coordinates around $z\in X$, such that $D$ is defined by $(x_1\cdots x_r=0)$. For every $p\geq 0$, the ideal $I_p(D)$ is generated around $z$ by \begin{equation}\label{hisnc}\{x_1^{a_1}\cdots x_r^{a_r} : \ 0\leq a_i\leq p, \ \sum{a_i} = p(r-1)\}\end{equation} \cite{hodgeideals}*{Proposition 8.2}. Weighted Hodge ideals of $D$ admit a similar local description.

\begin{proposition}\label{propositionsnc} Let $x_1, \ldots , x_n$ be coordinates around $z\in X$, such that $D$ is defined by $(x_1\cdots x_r=0)$. Then, for every $p\geq 0$ and $l\leq r$, $I_p^{W_l}(D)$ is generated around $z$ by $$\{ x_{j_1}^{a_1}\cdots x_{j_l}^{a_l}x_{I\setminus J}^{p+1} : \ J=\{j_1, \ldots , j_l\}\subseteq I, \ 0\leq a_i\leq p \text{ and } \sum{a_i} = p(l-1) \},$$ where $I = \{1,\ldots , r\}$. For $l\geq r$, $I_p^{W_l}(D) = I_p(D)$ around $z$. \end{proposition}

\begin{proof} The Hodge filtration of $W_{n+l}\OX(*D)$ also admits a simple description: \begin{equation}\label{wsncgen0} F_pW_{n+l}\OX(*D) = F_p\Dmod_X\cdot F_0W_{n+l}\OX(*D).\end{equation} Indeed, this follows from the fact that $\gr^W_{n+l}\OX(*D) \cong i_{l+}\shO_{E(l)}$ with a Tate Twist, so that the analogous statement of $(\ref{wsncgen0})$ is true for $i_{l+}\shO_{E(l)}$ (see e.g. \cite{saito09}*{Remark 1.1 iii}.)\\

For the rest of the proof we use right $\Dmod$-modules. By \cite{olano22a}*{Proposition 4.1}, $$F_0W_{n+l}\omega_X(*D) = W_l\omega_X(D).$$ Around $z$, $W_l\omega_X(D)$ is generated by $$\left\{\frac{\omega}{x_J}\right\}_{J\subseteq I, \ |J|=l}$$ where $\omega$ is the standard generator of $\omega_X$. It is clear that $W_l\omega_X(D)\cdot F_p\Dmod_X$ is generated by $$\left\{\frac{\omega}{x_{j_1}^{1+b_1}\cdots x_{j_l}^{1+b_l}} : \ \sum{b_i} = p, \ J\subseteq I, \text{ and }\ |J|=l\right\}.$$ The result follows from the equation $\frac{\omega}{x_{j_1}^{1+b_1}\cdots x_{j_l}^{1+b_l}} = \frac{\omega}{x_I^{p+1}} (x_{j_1}^{p-b_1}\cdots x_{j_l}^{p-b_l}x_{I\setminus J}^{p+1})$. The last statement follows from the fact that, if $l>r$, around $z$, $W_l\omega_X(D) = \omega_X(D)$.
\end{proof}

\subsection{Birational definition} Let $X$ be a smooth variety and $D$ a reduced divisor. Consider a log-resolution $f:Y\to X$ of the pair $(X,D)$, which is an isomorphism over $X\setminus D$, and denote $E=(f^*D)_{\rm red}$. A birational definition is given for Hodge ideals in \cite{hodgeideals}*{\textsection{9}}. In this section, we give a similar equivalent definition for weighted Hodge ideals. For the rest of this section we use right $\Dmod$-modules, as it is more convenient for the construction. Recall that the right $\Dmod_X$-module corresponding to $\OX(*D)$ is $\omega_X(*D)$, and $$F_{p-n}\omega_X(*D) = I_p(D) \otimes\omega_X((p+1)D).$$

Consider the following complex which we denote by $A^{\bullet}$: $$0\to f^*\Dmod_X\to \Omega^1_Y(\log{E})\otimes_{\shO_Y} f^*\Dmod_X\to \cdots\to \omega_Y(E)\otimes_{\shO_Y}f^*\Dmod_X\to 0$$ placed in degrees $-n, \ldots, 0$. The results in \cite{hodgeideals}*{\textsection{3}} say that the complex $A^{\bullet}$ represents the object $\omega_Y(*E)\Ltensor_{\Dmod_Y}\Dmod_{Y\to X}$ in the derived category of filtered right $f^{-1}\Dmod_X$-modules. Moreover, $R^0f_*A^{\bullet} \cong \omega_X(*D)$. \\

For $p\geq 0$ define the subcomplex $C^{\bullet}_{p-n}= F_{p-n}A^{\bullet}$ of $A^{\bullet}$ by $$0\to f^*F_{p-n}\Dmod_X\to \Omega^1_Y(\log{E})\otimes_{\shO_Y} f^*F_{p-n+1}\Dmod_X\to \cdots\to \omega_Y(E)\otimes_{\shO_Y}f^*F_p\Dmod_X\to 0.$$ The pushforward of this complex admits the following interpretation: $$R^0f_*C^{\bullet}_{p-n} = F_{p-n}\omega_X(*D) = I_p(D) \otimes\omega_X((p+1)D)$$ by \cite{hodgeideals}*{Remark 9.3, Corollary 12.1}.\\

We prove similar results in order to obtain a birational definition. Consider the complex $B^{\bullet}$: $$0\to \Dmod_Y\to \Omega^1_Y(\log{E})\otimes_{\shO_Y} \Dmod_Y\to \cdots\to \omega_Y(E)\otimes_{\shO_Y}\Dmod_Y\to 0$$ in degrees $-n,\ldots ,0$, where the map $$\Omega^p_Y(\log{E})\otimes\Dmod_Y \stackrel{d'}{\to} \Omega^{p+1}_Y(\log{E})\otimes\Dmod_Y$$ is given by $\omega\otimes P \to d\omega \otimes P + \sum{(dz_i\wedge \omega)\otimes \partial_i P}$. The complex $B^{\bullet}$ is filtered quasi-isomorphic to the object $\omega_Y(*E)$ in degree 0 \cite{hodgeideals}*{Proposition 3.1}.

\begin{proposition}\label{resolutionmodule} The complex $$W_lB^{\bullet}=0\to \Dmod_Y\to W_l\Omega^1_Y(\log{E})\otimes_{\shO_Y} \Dmod_Y\to \cdots\to W_l\omega_Y(E)\otimes_{\shO_Y}\Dmod_Y\to 0$$ in degrees $-n,\ldots, 0$ is quasi-isomorphic to $W_{n+l}\omega_Y(*E)$.
\end{proposition}

\begin{proof} We see first that the complex $W_lB^{\bullet}$ is exact in degrees $-n, \ldots, -1$. Fix a degree $-p$. We need to see that $$W_l\Omega_Y^{n-p-1}(\log{E})\otimes\Dmod_Y \to W_l\Omega_Y^{n-p}(\log{E})\otimes\Dmod_Y \stackrel{b}{\to} W_l\Omega_Y^{n-p+1}(\log{E})\otimes\Dmod_Y$$ is exact. Let $x\in X$ be a point and $\{z_1, \ldots , z_n\}$ be a set of coordinates in an open neighborhood around the point. We localize at $x$, take the completion, and identify the completion of $\shO_{X,x}$ with $\CC\llbracket z_1,\ldots , z_n\rrbracket$. Let $\eta\in\ker{\hat{b}}$. By exactness of $B^{\bullet}$, there exists $\omega$ in the completion of $\Omega_Y^{n-p-1}(\log{E})\otimes\Dmod_Y$ such that $d'\omega = \eta$ (we keep calling $d'$ the differentials of this complex). We can write $\omega = \sum{g_{I, J, \alpha} \frac{dz_I}{z_I}\wedge dz_J \otimes \partial^{\alpha}}$, with $g_{I,J,\alpha}\in\CC\llbracket z_1,\ldots ,z_n\rrbracket$, since every element $P \in \Dmod_Y$ can be written as $P = \sum{g_{\alpha}\partial^{\alpha}}$. Moreover, expanding each $g_{I, J, \alpha}$, we can write $$\omega = \sum{C^{\beta}_{I,J,\alpha}z^{\beta} \frac{dz_I}{z_I}\wedge dz_J \otimes \partial^{\alpha}}$$ so that no $z_i$ that appears in $z_I$ divides $C^{\beta}_{I,J,\alpha}z^{\beta}$. From this description, it follows that for each summand $C^{\beta}_{I,J,\alpha}z^{\beta} \frac{dz_I}{z_I}\wedge dz_J \otimes \partial^{\alpha}$, $|I|$ determines the weight where the form $C^{\beta}_{I,J,\alpha}z^{\beta} \frac{dz_I}{z_I}\wedge dz_J$ lies.\\

Next, we write, $\omega = \omega_{\leq l} + \omega_{>l}$, where the first term consists of the summands with $|I|\leq l$, and the latter of the terms with $|I|>l$. Using the description of $d'$, we see that $d'\omega_{\leq l}$ is in the completion of $W_l\Omega_Y^{n-p}(\log{E})\otimes\Dmod_Y$, and each summand of $d'\omega_{>l}$ is not. Indeed, \begin{equation*}\begin{split} & d'(C^{\beta}_{I,J,\alpha}z^{\beta} \frac{dz_I}{z_I}\wedge dz_J \otimes \partial^{\alpha}) \\ = & \sum_k{C^{\beta}_{I,J,\alpha}\beta_k z^{\beta-e_k}dz_k\wedge \frac{dz_I}{z_I}\wedge dz_J \otimes \partial^{\alpha}} + \sum_k{dz_k\wedge (C^{\beta}_{I,J,\alpha}z^{\beta} \frac{dz_I}{z_I}\wedge dz_J) \otimes \partial_k\partial^{\alpha}} \\ = & \sum_k{((-1)^{|I|}C^{\beta}_{I,J,\alpha}\beta_k)\ z^{\beta-e_k}\ \frac{dz_I}{z_I}\wedge (dz_k\wedge dz_J) \otimes \partial^{\alpha}} \\ & + \sum_k{((-1)^{|I|}C^{\beta}_{I,J,\alpha})\ z^{\beta}\  \frac{dz_I}{z_I}\wedge (dz_k\wedge dz_J) \otimes \partial_k\partial^{\alpha}}.\end{split}\end{equation*} Since $\eta\in\ker{\hat{b}}$, $d'\omega_{>l} = 0$, and $d'\omega_{\leq l} = \eta$, with $\omega_{\leq l}$ in the completion of $W_l\Omega_Y^{n-p-1}(\log{E})\otimes\Dmod_Y$.\\

Consider now the map, $$W_l\omega_Y(E)\otimes_{\shO_Y}\Dmod_Y\to W_{n+l}\omega_Y(*E)$$ given by $\frac{\omega}{f}\otimes P \to \frac{\omega}{f}\cdot P$. Fixing a degree of the Hodge filtration, and using the description of the Hodge filtration of $W_{n+l}\omega_Y(*E)$ (see for example \propositionref{propositionsnc}), we see that this map is surjective. That the kernel is the image of $W_l\Omega^{n-1}_Y(\log{E})\otimes_{\shO_Y} \Dmod_Y$ follows from \cite{hodgeideals}*{Proposition 3.1} and an argument similar to the one above.

\end{proof}

Consider next the complex $$W_lA^{\bullet} = 0\to f^*\Dmod_X\to W_l\Omega^1_Y(\log{E})\otimes_{\shO_Y} f^*\Dmod_X\to \cdots\to W_l\omega_Y(E)\otimes_{\shO_Y}f^*\Dmod_X\to 0.$$ We have that $W_lA^{\bullet} = W_lB^{\bullet}\otimes_{\Dmod_Y}\Dmod_{Y\to X}$, where $\Dmod_{Y\to X}= \shO_Y\otimes_{f^{-1}\OX}f^{-1}\Dmod_X$ is the transfer module. Note that when we see it as an $\shO_Y$ module, we simply write $f^*\Dmod_X$. 

\begin{lemma}\label{lemmatorsion} The complex $W_lA^{\bullet}$ represents $W_{n+l}\omega_Y(*E) \Ltensor_{\Dmod_Y} \Dmod_{Y\to X}$ in the derived category of filtered right $f^{-1}\Dmod_X$-modules. \end{lemma}

\begin{proof} It is enough to show that the elements $W_lB^k$ are acyclic with respect to $-\otimes_{\Dmod_Y}\Dmod_{Y\to X}$.
%Consider the spectral sequence $$E_1^{p,q} = \shTor_{-q}^{\Dmod_Y}(W_lB^p, \Dmod_{Y\to X}) \Rightarrow \shTor_{-(p+q)}^{\Dmod_Y}(W_{n+l}\omega_Y(*E), \Dmod_{Y\to X})$$ \cite{stacks}*{\href{https://stacks.math.columbia.edu/tag/061Y}{Tag 061Y}}. If $q\neq 0$, then $\shTor_{-q}^{\Dmod_Y}(W_lB^p, \Dmod_{Y\to X})=0$. 
For any $k$ consider the following spectral sequence: $$E_2^{p,q} = \shTor_p^{\Dmod_Y}(\shTor_q^{\OY}(W_l\Omega^k_Y(\log{E}), \Dmod_Y), \Dmod_{Y\to X}) \Rightarrow \shTor^{\OY}_{p+q}(W_l\Omega^k_Y(\log{E}), f^*\Dmod_{X})$$ \cite{weibel}*{Theorem 5.6.6}. As $\Dmod_Y$ is a locally free $\OY$-module, then $E_2^{p,q}=0$ for $q\neq 0$. Therefore, $$\shTor_p^{\Dmod_Y}(W_l\Omega^k_Y(\log{E})\otimes_{\shO_Y}\Dmod_Y, \Dmod_{Y\to X}) \cong \shTor^{\OY}_p(W_l\Omega^k_Y(\log{E}), f^*\Dmod_X)=0$$ for $p\neq 0$, where the last equality follows from the fact that $f^*\Dmod_X$ is locally free. \\

%This means that $\shTor_{-p}^{\Dmod_Y}(W_{n+l}\omega_Y(*E), \Dmod_{Y\to X})$ is isomorphic to the cohomology of $$W_l\Omega^{n+p-1}_Y(\log{E})\otimes_{\OY}f^*\Dmod_X \to W_l\Omega^{n+p}_Y(\log{E})\otimes_{\OY}f^*\Dmod_X\to W_l\Omega^{n+p+1}_Y(\log{E})\otimes_{\OY}f^*\Dmod_X.$$
\end{proof}

The map $$R^0f_*(W_{n+l}\omega_Y(*E) \Ltensor_{\Dmod_Y} \Dmod_{Y\to X}) \stackrel{\varphi}{\to} R^0f_*(\omega_Y(*E) \Ltensor_{\Dmod_Y} \Dmod_{Y\to X})$$ is precisely the morphism $$H^0f_+(W_{n+l}\omega_Y(*E)) \to H^0f_+(\omega_Y(*E)) = \omega_X(*D),$$ whose image is $W_{n+l}\omega_X(*D)$. Moreover, the complex $C_{p-n}^{\bullet}$ described above corresponds to the $F_{p-n}(\omega_Y(*E)\Ltensor_{\Dmod_Y} \Dmod_{Y\to X})$ using the identification $$\omega_Y(*E)\Ltensor_{\Dmod_Y} \Dmod_{Y\to X} \cong B^{\bullet}\otimes_{\Dmod_Y} \Dmod_{Y\to X}.$$ By strictness, there is an injective map $$R^0f_*C^{\bullet}_{p-n} \into R^0f_*A^{\bullet} \cong \omega_X(*D)$$ whose image is $F_{p-n}\omega_X(*D) = I_p(D) \otimes \omega_X((p+1)D)$ (see \cite{hodgeideals}*{Sections 4, 9, and 12}).\\

Similarly, we define $C^{\bullet}_{l,p-n}$ by $$ 0\to f^*F_{p-n}\Dmod_X\to W_l\Omega^1_Y(\log{E})\otimes_{\shO_Y} f^*F_{p-n+1}\Dmod_X\to \cdots\to W_l\omega_Y(E)\otimes_{\shO_Y}f^*F_p\Dmod_X\to 0$$ which corresponds to $F_{p-n}(W_{n+l}\omega_Y(*E)\Ltensor_{\Dmod_Y} \Dmod_{Y\to X})$ under the identification $$W_{n+l}\omega_Y(*E)\Ltensor_{\Dmod_Y} \Dmod_{Y\to X} \cong W_lB^{\bullet}\otimes_{\Dmod_Y} \Dmod_{Y\to X}$$ given by \lemmaref{lemmatorsion}.

\begin{proposition}\label{birationaldef} Using the notation above, $$I^{W_l}_p(D) \otimes \omega_X((p+1)D) = \im[R^0f_*C^{\bullet}_{l,p-n} \into R^0f_*W_lA^{\bullet} \to R^0f_*A^{\bullet}\cong \omega_X(*D)].$$
\end{proposition}

 \begin{proof} By strictness, we have an injective map $$R^0f_*C^{\bullet}_{l,p-n} \into R^0f_*W_lA^{\bullet}$$ whose image is $F_{p-n}H^0f_+W_{n+l}\omega_X(*D)$ (see for instance \cite{hodgeideals}*{\textsection{4}}). Taking the composition $$R^0f_*C^{\bullet}_{l,p-n} \into R^0f_*W_lA^{\bullet} \to R^0f_*A^{\bullet}\cong \omega_X(*D),$$ and using strictness in the middle morphism (since it underlies a morphism of mixed Hodge modules), the image corresponds to $F_{p-n}W_{n+l}\omega_X(*D) = I^{W_l}_p(D) \otimes \omega_X((p+1)D).$
\end{proof}

The description in \propositionref{birationaldef} for $I_0^{W_l}(D)$ coincides with the description in \cite{olano22a}*{Proposition 3}, since $f_*W_l\omega_Y(E) \to f_*\omega_Y(E)$ is an inclusion. The complex $C^{\bullet}_{l,1-n}$ also has a simple description. Recall that by definition $$C^{\bullet}_{l,1-n} = [W_l\Omega^{n-1}_Y(\log{E}) \to W_l\omega_Y(E)\otimes f^*F_1\Dmod_X]$$ in degrees -1 and 0. Moreover, the map $$\Omega^{n-1}_Y(\log{E}) \to \omega_Y(E)\otimes f^*F_1\Dmod_X$$ is injective \cite{hodgeideals}*{Lemma 3.4}. Using the fact that $W_l\Omega^{n-1}_Y(\log{E}) \into \Omega^{n-1}_Y(\log{E})$ and $W_l\omega_Y(E)\otimes f^*F_1\Dmod_X \into \omega_Y(E)\otimes f^*F_1\Dmod_X$ are injective (since $F_1\Dmod_X$ is a locally free $\OX$-module), we obtain that the differential in $C^{\bullet}_{l,1-n}$ is also an inclusion. Let $\mathcal{F}_{l,1}$ be the cokernel. This means that $$I_1^{W_l}(D) \otimes \omega_X(2D) = \im[f_*\mathcal{F}_{l,1} \to \omega_X(D)].$$ This map can be interpreted by using the complex $C^{\bullet}_{1-n}$. Indeed, let $\mathcal{F}_1$ be the cokernel of the differential in $C^{\bullet}_{1-n}$. We have an induced map $\mathcal{F}_{l,1} \to \mathcal{F}_1$. Since $f_*\mathcal{F}_1 = I_1(D)\otimes\omega_X(2D)$, $$I_1^{W_l}(D) \otimes \omega_X(2D) = \im[f_*\mathcal{F}_{l,1} \to f_*\mathcal{F}_1].$$

Note that since weighted Hodge ideals were defined in terms of the Hodge and weight filtrations of $\OX(*D)$, the constructions presented in this section are independent of the resolution of singularities.

\subsection{Weighted Hodge ideals and $V$-filtration}
Let $X$ be a smooth variety and $D$ be an effective reduced divisor defined by the global equation $f\in\OX(X)$. The Hodge ideals $I_p(D)$ can be described using the $V$-filtration of $i_+\OX$, where $i$ is the graph embedding defined by $f$. Namely, \begin{equation}\label{hivfiltration} I_p(D) = \left\{ \sum_{j=0}^p Q_j(1)f^{p-j}v_j \ : \  \sum_{j=0}^p v_j\partial_t^j\delta\in V^1i_+\OX\right\},\end{equation} where $Q_j(x) = \prod_{i=0}^{j-1}(x+i)$, \cite{mustatapopa20b}*{Theorem A'}. An equivalent description is obtained using the following map: $$\tau: V^1i_+\OX \to \OX(*D)$$ given by $$\tau\left(\sum_{i=0}^p v_i\partial_t^i\delta\right) = \sum_{i=0}^p Q_i(1)\frac{v_i}{f^{i+1}}.$$ The map $\tau$\footnote{The map $\tau$ corresponds to $\tau_1$ in the notation of \cite{mustatapopa20}. See \textsection{1} for the discussion about the reduced case.} is a surjective morphism of $\Dmod_X$-modules, and \begin{equation}\label{hodgevfiltration} I_p(D)\otimes\OX((p+1)D) = F_p\OX(*D) = \tau(F_{p+1}V^1i_+\OX),\end{equation} see \cite{mustatapopa20}*{Proposition 5.4 and Lemma 5.1}. Moreover, the map $\tau$ induces a map $$\bar{\tau}: \gr^1_Vi_+\OX \to \OX(*D)/\OX.$$ Indeed, it is enough to see that $\tau(V^{>1}i_+\OX ) \subseteq \OX$. This follows from the fact that $V^{>1}i_+\OX = V^{1+\alpha}i_+\OX = t\cdot V^{\alpha}i_+\OX$, with $\alpha >0$, and that if $j>0$, $tu\partial_t^j\delta = fu\partial_t^j\delta - ju\partial_t^{j-1}\delta$, and $tu\delta = fu\delta$. For $v=\sum_{j=0}^p v_j\partial_t^j\delta\in V^{>1}i_+\OX$, there exists $u = \sum_{j=0}^p u_j\partial_t^j\delta\in V^{\alpha}i_+\OX$ such that, $tu = v$. Hence, $$\tau(v) = \tau(fu_0\delta + \sum_{j=1}^p (fu_j\partial_t^j\delta - ju_j\partial_t^{j-1}\delta)) = u_0,$$ as $$Q_j(1) \frac{fu}{f^{j+1}} - jQ_{j-1}(1)\frac{u}{f^j} = 0$$ because $Q_j(1) = jQ_{j-1}(1)$.\\

The $\Dmod_X$-module $\gr^1_Vi_+\OX$ underlies the mixed Hodge module $\psi_{f,1}\OX$ and its weight filtration can be described in terms of the nilpotent operator $t\partial_t$. In order to complete the description in \theoremref{whivfiltration}, we first need to show that the map $\bar{\tau}$ also preserves the weight filtration.

\begin{proposition}\label{grvtau} The map $\bar{\tau}$ sends the weight and Hodge pieces to the same image as the map $\tau_{\Dmod_X}$ that underlies a morphism of mixed Hodge modules $$\tau^H: \psi_{f,1}\OX(-1) \to \cohH^1_D(\OX).$$ \end{proposition}

\begin{proof} The map $\bar{\tau}$ is surjective and using its description, we observe that its kernel is the image of the map $\partial_t t -1$ on $\gr^1_Vi_+\OX$. The same is true for the map $\tau_{\Dmod_X}$. Indeed, the map $\partial_t t -1$ underlies the composition $Var\circ can$ on $\psi_{f,1}\OX$. As $can: \psi_{f,1}\OX \to \phi_{f,1}\OX$ is surjective because $i_+\OX$ has strict support (see for instance \cite{schnellsurvey}*{\textsection 11}), the cokernel of $Var\circ can$ coincides with the cokernel of $$Var: \phi_{f,1}\OX \to \psi_{f,1}\OX(-1).$$ The cokernel of $Var$ is isomorphic to $i_{D*}\cohH^1i_D^!\OX$, where $i_D: D \to X$ is the inclusion \cite{saito90}*{Corollary 2.24}. Moreover, $i_{D*}\cohH^1i_D^!\OX$ is isomorphic to $\cohH^1_D(\OX)$ \cite{saito09}*{\textsection 2.2}. This means that $\bar{\tau}$ and $\tau_{\Dmod_X}$ could only differ by a $\Dmod_X$-automorphism of $\cohH^1_D(\OX)$ and the result is a consequence of \lemmaref{lemmaautomorphism}.
\end{proof}

\begin{lemma}\label{lemmaautomorphism} A $\Dmod_X$-automorphism of $\cohH^1_D(\OX)$ preserves the Hodge and weight filtration.
\end{lemma}

\begin{proof} We can restrict to an open affine subset. Let $X = \Spec{R}$ where $D$ is defined by $f\in R$, and $\varphi$ an $\Dmod_R$-automorphism of $R_f/R$. Let $m\geq 2$, then $\varphi[\frac{1}{f^m}] = [\frac{g_m}{f^m}]$ for some $g_m\in R$, since $f^m \varphi[\frac{1}{f^m}] = 0$. Using that $\varphi$ is $\Dmod_R$-linear, we see that for every $T\in Der_{\CC}(R)$, $T(g_m)\in (f^{m-1})$. This implies that around each smooth point of $P\in D$, using a regular system of parameters, we have an $h$ such that $h(P)\neq 0,$ and $g_m - g_m(P) \in f^m\cdot R_h$. Restricting the automorphism to the open set defined by $h$, we see that $\varphi_h$ acts by multiplying by a constant. This means that this constant doesn't depend on $m$, and after restricting to double intersections, we see that this constant doesn't depend on the point. Let $\lambda$ be the constant. Since $\varphi - \lambda\cdot Id$ is 0 on all the smooth points, $\varphi = \lambda\cdot Id$ everywhere. In particular, $\varphi$ preserves the Hodge and the weight filtration.
\end{proof}

We are very grateful to Mircea Musta\c{t}\u{a} for suggesting the argument of \lemmaref{lemmaautomorphism}.

\begin{remark}\label{remarkreferee} A simpler proof of the Lemma was suggested by a referee. Using the Riemann-Hilbert correspondence, and Verdier duality, it is enough to verify the conclusion of the Lemma on the perverse sheaf $\QQ_D[n-1]$. We leave the original proof to have an argument using only $\Dmod$-modules.
\end{remark}

A consequence of the result above, is that we can write a description of the weighted Hodge ideals in a similar way to (\ref{hivfiltration}). Let $W_lV^1i_+\OX$ be the submodule of $V^1i_+\OX$ which maps to $W_{n+l-2}\gr^1_Vi_+\OX$ via the canonical projection. 

\begin{proposition}\label{whivfiltration} Using the notation above, $$I^{W_l}_p(D) = \left\{ \sum_{j=0}^p Q_j(1)f^{p-j}v_j \ : \  \sum_{j=0}^p v_j\partial_t^j\delta\in W_lV^1i_+\OX\right\}.$$ \end{proposition}

\begin{proof} It follows from \propositionref{grvtau} that $\tau(F_{p+1}W_lV^1i_+\OX)=F_pW_{n+l}\OX(*D)= I_p^{W_l}(D)\otimes\OX((p+1)D).$ \end{proof}

The result above can be simplified even more using the description of the weight filtration of $\psi_{f,1}\OX$, and that is the statement of \theoremref{whivfiltration2}.

\begin{proof}[Proof of \theoremref{whivfiltration2}] First, we note that if $v\in V^1i_+\OX$, then $\tau(t\partial_t v) = 0$. Indeed, let $v=\sum_{j=0}^p v_j\partial_t^j\delta\in V^{1}i_+\OX$, then $$t\partial_t v = \sum_{j=0}^p (fv_j\partial_t^{j+1}\delta - (j+1)v_j\partial_t^j\delta),$$ and $$\tau(t\partial_t v) = \sum_{j=0}^p\left(Q_{j+1}(1) \frac{fv_j}{f^{j+2}} - (j+1)Q_j(1)\frac{v_j}{f^{j+1}}\right) = 0.$$ The weight filtration of $\psi_{f,1}\OX$ admits the following description for $k\geq 0$: $$W_{n-1+k}\psi_{f,1}\OX = \sum_{m\geq 0} (t\partial_t)^m(\ker{(t\partial_t)^{2m+k+1}})$$ (see \cite{saito94}*{2.7} and for the monodromy filtration see e.g. \cite{variationsofmhsi}*{Remark 2.3}).  The only piece that is not an image of $(t\partial_t)$ is $\ker{(t\partial_t)^{k+1}}$. That means that the subset $\ker{(t\partial_t)^{l}}\subseteq W_lV^1i_+\OX$ has the same image as $W_lV^1i_+\OX$ via $\tau$.
\end{proof}

\begin{remark}\label{qhomog} Let $(X,D)$ be a pair such that $D$ has at most isolated weighted homogeneous singularities. \theoremref{whivfiltration2} gives a complete description of the weighted Hodge ideals using the description of the $V$-filtration in \cite{saito09}. Using the notation above, in this case, $(t\partial_t)^2u\in V^{>1}i_+\OX$ for all $u\in V^1i_+\OX$. For this reason, $I_p^{W_2}(D) = I_p(D)$, for all $p\geq 0$. An argument without the use of the $V$-filtration in the case of $p=0$ is described in \cite{olano22a}*{\textsection 10}.
\end{remark}

A direct application of \theoremref{whivfiltration2} is that we can recover the following result proved in \cite{hodgeideals}*{Theorem C}. The proof we give differs from the one in \cite{hodgeideals} and is also much shorter. 

\begin{corollary}\label{corollaryadjoint} Let $X$ be a smooth variety and $D$ an effective reduced divisor. Then $$I_p(D) \subseteq \adj(D)$$ for all $p\geq 1$. \end{corollary}

\begin{proof} Recall that $\adj(D) = I_0^{W_1}(D)$ \cite{olano22a}*{Theorem A}. Moreover, as $I_p(D) \subseteq I_1(D)$ \cite{hodgeideals}*{Proposition 13.1}, it is enough to prove that $I_1(D)\subseteq I_0^{W_1}(D).$ \\

Let $u\in I_1(D)$. By (\ref{hivfiltration}), $u = u_0f + u_1$, where $f$ is defining equation of $D$, and $u_0\delta + u_1\partial_t\delta \in V^1i_+\OX$. We also have that $$V^1i_+\OX \ni (f-t)(u_0\delta + u_1\partial_t\delta) = u_1\delta,$$ and $$\partial_t ( u_1\delta) = u_1\partial_t\delta + u_0\delta - u_0\delta \in V^{>0}i_+\OX.$$ Finally, as $\delta\in V^{>0}i_+\OX$, then $u_0f\delta = t(u_0\delta)\in V^{>1}i_+\OX\subseteq W_1V^1i_+\OX$. This means that $u_0f\delta + u_1\delta \in W_1V^1i_+\OX$, hence $u_0f +u_1 \in I_0^{W_1}(D).$
\end{proof}

There is a relation between the minimal exponent of $f$ and the weighted Hodge ideals. Recall that if we denote $b_f(s)$ the Bernstein-Sato polynomial, and $\widetilde{b}_f(s)$ the reduced one, we call $\widetilde{\alpha_f}$ the negative of the largest root of $\widetilde{b}_f(s)$. Saito proved in \cite{saito16} that $I_p(D) = \OX$ if and only if $\widetilde{\alpha_f} \geq p+1$ (c.f. \cite{mustatapopa20b}*{Corollary 6.1}). Moreover, this result also holds in the case of $\QQ$-divisors, and it can be stated in the following form.

\begin{lemma}[\cite{mustatapopa20}*{Lemma 1.2}]\label{lemma1.2} For an integer $p$ and $\alpha\in (0,1]$, $$\partial_t^p\delta\in V^{\alpha}i_+\OX \Leftrightarrow \widetilde{\alpha_f} \geq p+\alpha.$$ 
\end{lemma}

Using these ideas, we obtain the following result for the 1st weighted Hodge ideals.

\begin{corollary}\label{whiminimalexp} Using the notation above, $$I_p^{W_1}(D) = \OX \text{ if and only if } \widetilde{\alpha_f} > p+1.$$
\end{corollary}

\begin{proof} Suppose first that $\widetilde{\alpha_f} > p+1$. Then, by \lemmaref{lemma1.2}, $\partial_t^p\delta\in V^1i_+\OX$. Moreover, there exists $\alpha\in (0,1]$ such that $\widetilde{\alpha_f} \geq p+1+\alpha$. Again, by \lemmaref{lemma1.2}, $\partial_t^{p+1}\delta\in W_1V^1i_+\OX$, and therefore, $I_p^{W_1}(D) = \OX$.\\

Suppose now that $I_p^{W_1}(D) = \OX$. Then, $I_p(D) = \OX$, and in particular $\delta, \partial_t\delta, \ldots, \partial_t^p\delta\in V^1i_+\OX$. Moreover, there exists $v = \sum_{j=0}^p v_j\partial_t^j\delta\in W_1V^1i_+\OX$ such that $\sum_{j=0}^p Q_j(1)f^{p-j}v_j = 1$. It is enough to show that $\partial_t^p\delta\in W_1V^1i_+\OX$. Indeed, by \propositionref{whivfiltration} and the injectivity of $t: \gr_V^0i_+\OX \to \gr_V^1i_+\OX$ (see e.g. \cite{schnellsurvey}*{\textsection{11}}), this means that $\partial_t^{p+1}\delta \in V^{\alpha}i_+\OX$ with $\alpha\in (0,1]$, and therefore, $\widetilde{\alpha_f} \geq p+1+\alpha > p+1$. We argue by induction. Suppose $p=0$. Then $v=v_0\delta$ and by the second condition, $v_0 = 1$. Hence, $\delta\in W_1V^1i_+\OX$. By the induction hypothesis, we assume now that $\partial_t^k\delta\in W_1V^1i_+\OX$ for $k=0, \ldots , p-1$. It follows from the description of $v$ that $$Q_p(1)v_p = 1 - f(\sum_{j=0}^{p-1} Q_j(1)f^{p-1-j}v_j),$$ and then $$ v = \partial_t^p\delta - f(\sum_{j=0}^{p-1} Q_j(1)f^{p-1-j}v_j)\partial_t^p\delta + \sum_{j=0}^{p-1} v_j\partial_t^j\delta .$$ The result follows if we show that $f\partial_t^p\delta\in W_1V^1i_+\OX$, and this is a consequence of $f\partial_t^p\delta = t\partial_t(\partial_t^{p-1}\delta) + p\partial_t^{p-1}\delta \in W_1V^1i_+\OX$. 

\end{proof}

\begin{remark} In general, we cannot obtain more information about the other $p$-weighted Hodge ideals. In \cite{olano22a}*{\textsection 13}, the case of isolated log-canonical singularities, that are not rational, is discussed. This case corresponds to $\widetilde{\alpha_f} = 1$. By the discussion above, it is clear that $I_0(D) = \OX$ and that $I_0^{W_1}(D)$ is not trivial. For $l = 2,\ldots, n-1$, there are examples of $f$ where the weighted multiplier ideals $I_0^{W_l}(D)$ are trivial, and other examples where they are non-trivial \cite{ishii85}*{Theorem 5.2}.
\end{remark}

\section{Local study} 

\subsection{Measuring the difference between weighted Hodge ideals.}\label{sectiondifference} There is a short exact sequence that arises from the definition of the weight filtration on $\OX(*D)$: $$0\to W_{n+l-1}\OX(*D) \to W_{n+l}\OX(*D) \to \gr^W_{n+l}\OX(*D) \to 0.$$ Applying $F_p$, we obtain the short exact sequence \begin{equation}\label{differencesequence} 0\to I_p^{W_{n+l-1}}(D) \otimes \OX((p+1)D) \to I_p^{W_{n+l}}(D) \otimes \OX((p+1)D) \to F_p \gr^W_{n+l}\OX(*D) \to 0.\end{equation} When $D$ has at most isolated singularities and $l\geq 2$, $\gr^W_{n+l}\OX(*D)$ is supported on the singular points. To simplify the notation, we use the following definition.

\begin{definition} Suppose $D$ has at most one isolated singularity $x\in D$, and let $i_x:\{x\}\into X$. For $l\geq 2$, we denote by $H_l$ the complex pure Hodge structure of weight $n+l$ such that $$\gr^W_{n+l}\OX(*D) \cong (i_{x})_+H_l.$$
\end{definition}

In order to describe the dimension of $F_p(i_{x})_+H_l$, it is enough to describe the dimension of $\Gr_F^{n-k}H_l$ for $0\leq k\leq p$. This is a consequence of the local description of the Hodge filtration of $(i_{x})_+H_l$. Let $x_1, \ldots, x_n$ be a set of coordinates around the point $x\in X$. We have the following description of the pushforward of $H_l$ as a $\Dmod$-module: \begin{equation} (i_{x})_+H_l = (i_{x})_*H_l \otimes_{\CC} \CC[\partial_1, \cdots, \partial_n], \end{equation} where $\partial_i = \frac{\partial}{\partial_{x_i}}$, and \begin{equation}\label{pointpushfwd} F_p(i_{x})_+H_l = \bigoplus_{\nu\in\ZZ_{\geq 0}^n} (i_{x})_*F_{p-|\nu|-n}H_l \otimes \partial^{\nu}, \end{equation} where $\partial^{\nu} = \partial_1^{\nu_1}\cdots\partial_n^{\nu_n}$, $|\nu| = \nu_1+\ldots +\nu_n$, and $F_kH_l = F^{-k}H_l$. Since the lowest degree of the Hodge filtration of $\OX(*D)$ is 0, and $\DR((i_{x})_+H_l) \cong (i_{x})_*H_l$, that is, the push-forward of the pure Hodge structure $H_l$ is a skyscraper sheaf, then the highest degree of the Hodge filtration of $H_l$ is $n$, in other words, $F^{n+1}H_l = 0$. Using this, we obtain, for instance, that $$F_0(i_{x})_+H_l = (i_{x})_*F^nH_l\otimes 1 = (i_{x})_*\Gr_F^nH_l\otimes 1,$$ and $$F_1(i_{x})_+H_l = (i_{x})_*F^{n-1}H_l\otimes 1 \oplus \bigoplus_i (i_{x})_*F^nH_l\otimes \partial_i.$$ Since $F_p(i_{x})_+H_l$ is a skyscraper sheaf, we denote by $\dim(F_p(i_{x})_+H_l)$ the dimension of the complex vector space $J_p$ that satisfies $F_p(i_{x})_+H_l = (i_{x})_*J_p$. From the discussion above, we obtain that  $$\dim(F_0(i_{x})_+H_l) = \dim(\Gr_F^nH_l),$$ $$\dim(F_1(i_{x})_+H_l) = \dim(F^{n-1}H_l) + n\dim(\Gr_F^nH_l) = \dim(\Gr_F^{n-1}H_l) + (n+1)\dim(\Gr_F^nH_l),$$ and in general \begin{equation}\label{dimensiondifference}\begin{split} \dim(F_p(i_{x})_+H_l) &= \sum_{k=0}^p \binom{n-1+k}{k}\dim(F^{n-p+k}H_l) = \sum_{r=0}^p\dim(\Gr_F^{n-r}H_l)\sum_{k=0}^{p-r}\binom{n-1+k}{k} \\ & = \sum_{r=0}^p\binom{n+p-r}{p-r}\dim(\Gr_F^{n-r}H_l). \end{split}\end{equation}

The dimension of $\Gr_F^{n-k}H_l$ is described in \theoremref{thmdifference}.

\begin{proof}[Proof of \theoremref{thmdifference}] We can and will assume that $X$ is a projective variety. Indeed, there is an open set around $x$ which has a smooth projective compactification $\bar{X}$. Let $\bar{D}$ be the closure of $D$ in $\bar{X}$. Consider a log-resolution of $(\bar{X}\smallsetminus x, \bar{D}\smallsetminus x)$ given by a sequence of blow ups with centers over the singular locus of $\bar{D}\smallsetminus x$. By blowing up the same sequence of centers over $\bar{X}$, we obtain a map $X_1\to \bar{X}$. Let $D_1$ be the strict transform of $\bar{D}$. By construction, the map is an isomorphism over $(X,D)$, and $D_1$ has only one isolated singularity corresponding to $x\in D$. We replace $(X,D)$ with $(X_1,D_1)$.\\

First, we prove that these dimensions do not depend on the log-resolution of singularities that is an isomorphism outside of $\{x\}$. Since for a pair of resolution of singularities one can find a third one that dominates the two of them, it is enough to show that the dimensions are equal if we have two resolutions of singularities $g_1:D_1\to D$ and $g_2:D_2\to D$ such that there is a morphism $h:D_1\to D_2$ such that $g_1 = g_2\circ h$. Let $G_i\subseteq D_i$ be the exceptional divisor of $g_i$. Consider the exact sequence of mixed Hodge structures $$\cdots \to H^{k-1}(G_1) \to H^k(D_2)\to H^k(D_1)\oplus H^k(G_2) \to H^k(G_1)\to \cdots$$ (see \cite{PS}*{Proof of Theorem 6.15}). For $l\geq 3$, applying $H^{p,n-l-p}$, we obtain that $$H^{p,n-l-p}(H^{n-2}(G_2)) \cong H^{p,n-l-p}(H^{n-2}(G_1)).$$ For $l=2$, applying $H^{p,n-p-2}$ and $H^{n-p-1,p+1}$, and noting that $h^{p,n-p-2}(D_i) = h^{n-p-1,p+1}(D_i)$, we obtain that $$h^{p,n-p-2}(H^{n-2}(G_1)) - h^{n-p-1,p+1}(H^n(G_1)) = h^{p,n-p-2}(H^{n-2}(G_2)) - h^{n-p-1,p+1}(H^n(G_2)).$$

Let $f:Y\to X$ be a log-resolution that is an isomorphism outside of $x$, and $E:=f^{-1}(D)_{red}$. This resolution defines a log-resolution of singularities $g: \widetilde{D}\to D$ by restriction, that is an isomorphism outside of $x$. We use the spectral sequence (\ref{ss1}) for the constant map from $X$ to a point. In this case, it says \begin{equation}\label{ss2} E_1^{-n-l,q} = \HH^{q-n-l}(X, \DR(\gr^W_{n+l}\OX(*D))) \Rightarrow H^{q-l}(U,\CC),\end{equation} noting that $\DR(\OX(*D)) \cong \textbf{R}j_*\CC_U[n]$, where $j:U = X\smallsetminus D\into X$. We also have the isomorphism $$E_2^{-n-l,q}\cong \Gr^W_{q}H^{q-l}(U).$$ Consider the maps $$E_1^{-n-l-1,n+l}\to E_1^{-n-l,n+l}\to E_1^{-n-l+1,n+l}, $$ corresponding to $$\HH^{-1}(X, \DR(\gr^W_{n+l+1}\OX(*D))) \to \HH^{0}(X, \DR(\gr^W_{n+l}\OX(*D))) \to \HH^{1}(X, \DR(\gr^W_{n+l-1}\OX(*D))).$$ Moreover, the degeneration of the Hodge-to-de-Rham spectral sequence says that \begin{equation}\label{hdrdeg}\gr^F_{-n+p}\HH^{i}(X, \DR(\gr^W_{n+l}\OX(*D))) \cong \HH^{i}(X, \gr^F_{-n+p}\DR(\gr^W_{n+l}\OX(*D)))\end{equation} (see for example \cite{hodgeideals}*{Example 4.2}).\\

Consider first the case $l\geq 3$. Noting that $\HH^{i}(X, \DR((i_{x})_+H_l)) = 0$ if $i\neq 0$ for $l\geq 2$, we obtain that $$E_2^{-n-l,n+l} \cong H_l.$$ Applying $\gr^F_{-n+p}$, using (\ref{hdrdeg}), and the $E_2$-degeneration of the spectral sequence, we obtain that $$ \gr^F_{-n+p}E_2^{-n-l,n+l} \cong \gr^F_{-n+p}H_l = \Gr_F^{n-p}H_l \cong H^{n-p,l+p}(H^n(U)) \cong H^{p, n-l-p}(H^n_c(U))^*,$$ where the last isomoprhism follows from Poincaré duality (see \cite{PS}*{Theorem 6.23}). Using the long exact sequence of the pair $(X,D)$, we obtain that $$H^{p,n-l-p}(H^n_c(U)) \cong H^{p,n-l-p}(H^{n-1}(D)),$$ as $H^{n-1}(X)$ and $H^n(X)$ have pure Hodge structures. Finally, as $g$ has $\{x\}$ as discriminant, we have a long exact sequence, $$H^{n-2}(\widetilde{D})\to H^{n-2}(G)\to H^{n-1}(D)\to H^{n-1}(\widetilde{D}).$$ As this is a sequence of mixed Hodge structures, we obtain $$H^{p,n-l-p}(H^{n-1}(D)) \cong H^{p,n-l-p}(H^{n-2}(G)).$$

Consider now $l=2$.  In this case, the maps $$E_1^{-n-3, n+2} \to E_1^{-n-2, n+2} \to E_1^{-n-1, n+2} \to E_1^{-n,n+2}$$ correspond to $$0 \to H_2 \stackrel{\widetilde{\beta}}{\to} H^n(D)(-1)\stackrel{\widetilde{\gamma}}{\to} H^{n+2}(X).$$ Indeed, the first two terms follow from the explanation above. The third term follows from the fact that $\DR(\gr^W_{n+1}\OX(*D)) \cong IC_D(-1)$, a Tate twist of the intersection complex of $D$ \cite{saito09}*{\textsection 2.2}. Furthermore, $IH^n(D) \cong H^n(D)$ \cite{goreskymacpherson80}*{\textsection 6.1}. The last term in the complex, follow as $\DR(\gr^W_{n}\OX(*D)) \cong \CC_X[n].$ From the short exact sequence $$\ker{\beta} \to \Gr^{n-p}_FH_2 \to \im{\beta},$$ where $\beta = \gr^F_{-n+p}\widetilde{\beta}$ and $\gamma = \gr^F_{-n+p}\widetilde{\gamma}$, we obtain that \begin{equation*}\begin{split}\dim(\Gr^{n-p}_FH_2) & = \dim{\ker{\beta}} + \dim{\im{\beta}} \\ &= h^{p, n-p-2}(H^n_c(U)) + h^{n-p-1,p+1}(H^n(D)) \\&- h^{p, n-p-2}(X) + h^{p, n-p-2}(H^{n-2}_c(U)) - h^{p, n-p-2}(H^{n-1}_c(U)).\end{split}\end{equation*} Indeed, this follows from the descriptions of $E_2^{n-2+s, n+2}$ for $s=0,1,2$ and Poincaré duality. More precisely, we have three short exact sequences $$ 0\to H^{p,n-p-2}(H^n_c(U))^*\to \Gr_F^{n-p}H_2 \to \im{\beta} \to 0,$$ $$0\to \ker{\gamma} \to \Gr_F^{n-p}H^{n+1}(D)\to \im{\gamma}\to 0,$$ $$0 \to \im{\gamma}\to H^{p,n-p-2}(X)^* \to H^{p,n-p-2}(H^{n-2}_c(U))^*\to 0,$$ and also that $\Gr_F^{n-p}E_2^{n-1,n+2}\cong H^{p,n-p-2}(H^{n-1}_c(U))^*$. Using the long exact sequence associated to the pair $(X,D)$ to relate these three sequences, we obtain $$ \dim(\Gr^{n-p}_FH_2) = h^{n-p-1, p+1}(H^n(D)) - h^{p, n-p-2}(H^{n-2}(D)) + h^{p, n-p-2}(H^{n-1}(D)).$$ Finally, using that the map $g$ has $\{x\}$ as discriminant, we obtain that $$\dim(\Gr^{n-p}_FH_2) = h^{p,n-p-2}(H^{n-2}(G)) - h^{n-p-1, p+1}(H^n(G)).$$
\end{proof}

\begin{remark}\label{rmksecondterm} In general, the term $h^{n-p-1, p+1}(H^n(G))$ might not be 0. Consider for instance $n=4$ and $p=1$. In this case, $h^{2,2}(H^4(G)) = k$ where $k$ is the number of irreducible components of $G$. Using similar computations as above, we also see that $$h^{p,n-p-2}(H^{n-2}(G)) - h^{n-p-1, p+1}(H^n(G)) = h^{n-p-1, p+1}(H^n(D)) - h^{p,n-p-2}(H^{n-2}(D)),$$ that is, the failure of Poincaré duality. Still, in the case $p=0$, the term $h^{n-p-1, p+1}(H^n(G))$ is always 0, as $G$ is $(n-2)$-dimensional (see \cite{olano22a}*{Theorem B}).
\end{remark}

\section{Vanishing Theorems}

\subsection{Ample divisors} Let $X$ be a smooth projective variety of dimension $n$, and $D$ an ample divisor. Let $U = X\smallsetminus D$. As $U$ is smooth and affine, $H^{i+n}(U) = 0$ for $i> 0$ (see for instance \cite{lazarsfeld2}*{Theorem 3.1.1}). In this setting we have the following result. 

\begin{lemma}\label{ampleisolated}There is a short exact sequence \begin{equation*}\begin{split} 0\to &  H^i(X, \DR(W_{n+l}\OX(*D)) \to H^i(X, \DR(\gr^W_{n+l}\OX(*D))) \to \\ \to & H^{i+1}(X,\DR(W_{n+l-1}\OX(*D)))\to 0.\end{split}\end{equation*}
\end{lemma}

\begin{proof} In \cite{olano22a}*{Proof of Proposition 12.1}, using the spectral sequences $$E_1^{-n-l, q} = H^{q-n-l}(X, \DR(\gr^W_{n+l}\OX(*D))) \Rightarrow H^{q-l}(U, \CC)$$ and $$ E_1^{'-n-l, q} = H^{q-n-l}(X, \DR(\gr^W_{n+l}W_{n+k}\OX(*D))) \Rightarrow H^{q-n-l}(X, \DR(W_{n+k}\OX(*D)))$$ and noting that $$E_{2}^{-n-l,q} \cong \Gr^W_qH^{q-l}(U, \CC)$$ $$E_{2}^{'-n-l,q} \cong \gr^W_qH^{q-n-l}(X, \DR(W_{n+k}\OX(*D)))$$ we obtained:
\begin{aenumerate}\item\label{a} For $s\geq 1$, $$\gr^W_{n+k+i-s}H^i(X, \DR(W_{n+k}\OX(*D))) \cong \Gr^W_{n+k+i-s}H^{i+n}(U, \CC).$$
\item\label{b} For $s\geq 1$, $$ \gr^W_{n+k+i+s}H^i(X, \DR(W_{n+k}\OX(*D))) = 0.$$
\item\label{c} Let $$\alpha_{k+1}: H^{i-1}(X, \DR(\gr^W_{n+k+1}\OX(*D))) \to H^i(X, \DR(\gr^W_{n+k}\OX(*D)))$$ corresponding to the map $E^{-n-k-1, i+n+k}_1 \to E_1^{-n-k, i+n+k}$. Then we have the following short exact sequence: $$0\to \im{\alpha_{k+1}} \to \gr^W_{i+n+k}H^i(X, \DR(W_{n+k}\OX(*D))) \to \Gr^W_{i+n+k}H^{i+n}(U, \CC) \to 0.$$
\end{aenumerate}
If $i\geq 1$, then $$\im{\alpha_{k+1}}\cong \gr^W_{i+n+k}H^i(X, \DR(W_{n+k}\OX(*D))) \cong H^i(X, \DR(W_{n+k}\OX(*D))).$$

Consider now the complex $$E_1^{-n-l-1, n+l+i} \stackrel{\alpha_{l+1}}{\to} E_1^{-n-l, n+l+i} \stackrel{\alpha_l}{\to} E_1^{-n-l+1, n+l+i}.$$ As $E_2^{-n-l,n+l+i} = 0$, using the analysis above, we obtain a short exact sequence $$0\to \im{\alpha_{l+1}} \to H^i(X, \DR(\gr^W_{n+l}\OX(*D))) \to \im{\alpha_l}\to 0,$$ and the result follows.
\end{proof}

When $p= 0$, the result above is enough to obtain that $$H^i(X, \omega_X(D) \otimes I_0^{W_l}(D)) = 0 $$ for $l\geq 2$ and $i\geq 1$. Indeed, as 0 is the lowest degree of the Hodge filtration on $\OX(*D)$, we have $$H^i(X, \omega_X(D) \otimes I_0^{W_l}(D)) \cong \gr_{-n}^FH^i(X, \DR(W_{n+l}\OX(*D))).$$ This is no longer the case when we consider $\gr^F_{-n+p}$ for $p\geq 1$ instead. Nonetheless, following the idea in \cite{hodgeideals}*{Proof of Theorem F}, we give conditions in \theoremref{vanishingample} to obtain an analogue vanishing theorem.

\begin{proof}[Proof of \theoremref{vanishingample}]
Since $I_{p-1}^{W_l}(D) = \OX$, we have the following short exact sequence $$0\to \omega_X(pD) \to \omega_X((p+1)D) \otimes I_p^{W_l}(D) \to \omega_X \otimes \gr^F_p(W_{n+l}\OX(*D)) \to 0.$$ Using the long exact sequence of cohomologies and Kodaira-vanishing, we note that it is enough to prove that $$H^i(X, \omega_X \otimes \gr^F_p(W_{n+l}\OX(*D))) = 0.$$

Consider now the complex $$C^{\bullet} := \gr^F_{-n+p}\DR(W_{n+l}\OX(*D)).$$ The complex $C^{\bullet}$ can be identified with the complex \begin{equation*}\begin{split}\left[\Omega_X^{n-p}\otimes \OX(D) \to \Omega_X^{n-p+1}\otimes \shO_D(2D) \to \cdots \to \Omega_X^{n-1}\otimes \shO_D(pD)\to \omega_X \otimes \gr^F_p(W_{n+l}\OX(*D))\right]
\end{split}\end{equation*}
concentrated in degrees $-p$ to $0$, since $F_0W_{n+l}\OX(*D) = \OX(D)$ and $\gr^F_kW_{n+l}\OX(*D) \cong \shO_D((k+1)D)$ for $k\leq p-1$ (see \textsection{\ref{mhmpreliminaries}} for the definition of $\gr^F_p\DR )$.\\

Suppose now that $D$ has at most isolated singularities. By \lemmaref{ampleisolated}, we obtain that $$\HH^i(X, \DR(W_{n+l}\OX(*D)))= 0$$ for $i\geq 1$ and $l\geq 2$. In particular, this means that $$\HH^i(X, C^{\bullet}) =0 $$ for the same indices, by the Hodge-to-de-Rham degeneration. Next, we use the exact sequence $$ E_1^{p,q} = H^q(X, C^p) \Rightarrow \HH^{p+q}(X, C^{\bullet}).$$ Note that $$E_1^{0,q} = H^q(X, \omega_X \otimes \gr^F_p(W_{n+l}\OX(*D))).$$
Since $$E_1^{-1, q} = H^q(X, \Omega_X^{n-1}\otimes \shO_D(pD)),$$ then $E_1^{-1, q} = 0$ if $q\geq 2$ by Nakano vanishing. Moreover, $E_1^{-1,1} = 0$ by our hypothesis.\\

We continue with a similar analysis in the higher pages of the spectral sequence. More precisely, we show that the hypothesis implies that $E_r^{-r, q+r-1} = 0$ for all $r\geq 2$. Note that this is enough to complete the proof. Indeed, if this is the case, we obtain that $$E_{\infty}^{0,q} = H^q(X, \omega_X \otimes \gr^F_p(W_{n+l}\OX(*D))) = 0$$ for $q\geq 1$, where the last equality follows from the established equality with $C^{\bullet}$.\\

To complete the proof, note that $$E_1^{-r, q+r-1} = 0 $$ for $r\geq p$. Indeed, this is clear for the strict inequality by the degrees on which $C^{\bullet}$ is concentrated, and $$E_1^{-p, q+p-1} = H^{q+p-1}(X, \Omega_X^{n-p}\otimes \OX(D)).$$ If $q\geq 2$, then this spaces vanishes by Nakano vanishing, and if $q=1$, it vanishes by our assumption. Finally, for $r\leq p-1$, we have $$E_1^{-r, q+r-1} = H^{q+r-1}(X, \Omega_X^{n-r}\otimes \shO_D((p+1 -r)D)).$$ This space fits the a long exact sequence \begin{equation*}\begin{split}  \to H^{q+r-1}(X, \Omega_X^{n-r}((p+1 -r)D)) \to E_1^{-r, q+r-1} \to H^{q+r}(X, \Omega_X^{n-r}((p -r)D)). \end{split}\end{equation*}
If $q\geq 2$, then the two other terms vanish by Nakano vanishing, and if $q = 1$, they vanish by the assumption.  
\end{proof}

\begin{remark}\label{rmkvanishingample} This result does not hold in general for $l = 1$ (see \cite{olano22a}*{Remark 9}).
\end{remark}

\subsection{Kodaira-type vanishing} Using a similar idea to the one in the proof of \theoremref{vanishingample}, we obtain a vanishing theorem for weighted Hodge ideals. This is the analogue result to \cite{hodgeideals}*{Theorem F}.

\begin{proposition}\label{vanishingkodaira} Let $X$ be a smooth projective variety of dimension $n$, and $D$ a reduced effective divisor. Let $L$ be a line bundle such that $L(kD)$ is ample for $0\leq k\leq p$, and assume $I_{p-1}^{W_1}(D)$ is trivial. Then
\begin{enumerate}\item For $l\geq 1$ and $i\geq 2$,  $$H^i(X, \omega_X((p+1)D) \otimes L \otimes I_p^{W_l}(D))= 0.$$ 
\item If $H^j(X, \Omega_X^{n-j} \otimes L((p-j+1)D)) = 0$ for all $1\leq j\leq p$, then $$H^1(X, \omega_X((p+1)D) \otimes L \otimes I_p^{W_l}(D))= 0$$ for $l\geq 1$.
\end{enumerate}
\end{proposition}

\begin{proof} Since $I_{p-1}^{W_l}(D) = \OX$, we have the following short exact sequence $$0\to \omega_X\otimes L(pD) \to \omega_X\otimes L((p+1)D) \otimes I_p^{W_l}(D) \to \omega_X \otimes L \otimes \gr^F_p(W_{n+l}\OX(*D)) \to 0.$$
By Kodaira-vanishing, it is enough to prove $$H^i(X, \omega_X \otimes L \otimes \gr^F_p(W_{n+l}\OX(*D))) = 0.$$ We have that $$\HH^i(X, L \otimes \gr^F_{-n+p}\DR(W_{n+l}\OX(*D))=0 $$ for $i\geq 1$ and $l\geq 1$ as a consequence of a vanishing result by Saito \cite{saito90}*{Proposition 2.33}. To complete the proof, we use the same spectral sequence as in the proof of \theoremref{vanishingample}. 
\end{proof}

\subsection{Applications} In this section, we combine the local study and the vanishing results. To obtain applications, we use the vanishing theorems of the previous sections. A class varieties where the vanishing condition in \theoremref{vanishingample} and \propositionref{vanishingkodaira} is satisfied, is toric varieties. In this case, the Bott-Danilov-Steenbrink vanishing theorem says that if $A$ is an ample line bundle on the toric variety $X$, then $$H^i(X, \Omega_j\otimes A)= 0 $$ for $j\geq 0$ and $i\geq 1$ (see e.g. \cite{mustata02}*{Theorem 2.4}). For the applications, we discuss the case of $X=\PP^n$. We start with the proof of \corollaryref{indconditions}.

\begin{proof}[Proof of \corollaryref{indconditions}] Consider the exact sequence $$0\to \shO_{\PP^n}(k) \otimes I_p^{W_l}(D) \to \shO_{\PP^n}(k) \to \shO_{Z_{p,l}}(k)\to 0.$$ The result follows from passing to cohomology and applying \theoremref{vanishingample}.
\end{proof}

\subsubsection{Isolated $p$-log-canonical singularities} Suppose the pair $(X,D)$ is $p$-log-canonical, and has at most isolated singularities. If $p=0$, the pair is log-canonical and in this case, $I_0^{W_1}(D)$ is the maximal ideal at each isolated singularity that is not rational, by a result of Ishii (see \cite{olano22a}*{\textsection 5.3}). For simplicity, let $x\in D$ be the only singularity and $i:\{x\}\into X$ the inclusion, and suppose that it is log-canonical singularity and not rational. The result above means that if we denote $$i_*H_l = \DR(\gr_{n+l}^W\OX(*D))$$ for $l\geq 2$, there exists exactly one degree $l$ such that $\dim(\gr^F_{-n}H_l) = 1$, and the rest are 0. In this case, using \cite{olano22a}*{Theorem B}, we say that the singularity is of type $(0,n-l)$ \cite{ishii85}*{Definition 4.1}. There is a similar picture for the cases $p\geq 1$ we describe next.\\

Non-rational log-canonical singularities correspond to the case where the minimal exponent at the singularity is 1. We then consider singularities with minimal exponent $p+1$, in which case $I_p(D) = \OX$ and $I_p^{W_1}(D)$ is non-trivial by \corollaryref{whiminimalexp}. These singularities generalize the example of non-rational log-canonical singularities in the following sense.

\begin{proposition}\label{minexprk1} Suppose $D$ has at most one isolated singularity $x\in D$, and $\widetilde{\alpha_D} = p+1$. Then, $$I_p^{W_1}(D) = \mathfrak{m}_x,$$ the maximal ideal of $x$ in $X$.
\end{proposition} 

\begin{proof} Suppose that $D$ is defined by $f\in\OX$. Recall from the proof of \corollaryref{whiminimalexp}, that as $\widetilde{\alpha_f} = p+1$, then $\delta, \partial_t\delta, \ldots, \partial_t^p\delta\in V^1B_f$. Moreover, we also know that $\delta, \partial_t\delta, \ldots, \partial_t^{p-1}\delta\in W_1V^1B_f$. It is then enough to show that $g\partial_t^p\delta\in W_1V^1B_f$ if and only if $g\in \mathfrak{m}_x$. As $D$ has an isolated singularity, we have that $$\gr^F_p\gr_V^{\alpha} B_f \text{ is annihilated by } \mathfrak{m}_x$$ for $\alpha <1$  \cite{dimcasaito12}*{4.11.1}.\\

We also know that $\partial_t^p\delta\in V^1B_f\smallsetminus W_1V^1B_f$, and this means that $\partial_t^{p+1}\delta \in V^0B_f\smallsetminus V^{>0}B_f$. In particular, the class of $\partial_t^{p+1}\delta$ in $\Gr^F_p\gr_V^{0}B_f$ is not zero. Using the result above, for any $g\in\mathfrak{m}_x$, the class of $g\partial_t^{p+1}\delta$ in $\Gr^F_p\gr_V^{0}B_f$ is zero. This means that $g\partial_t^{p+1}\delta\in V^{>0}B_f$, and equivalently, $g\partial_t^p\delta\in W_1V^1B_f$. Using the description of \theoremref{whivfiltration2}, we obtain that $g\in I_p^{W_1}(D)$ for any $g\in\mathfrak{m}_x$, and we know that the ideal is not trivial, hence we have an equality.
\end{proof}

In other words, if $D$ has one isolated singularity $x\in D$, and $\widetilde{\alpha_D} = p+1$, then $$\sum_{l\geq 2}{\dim(\Gr_F^{n-p}H_l)} = 1$$ by \theoremref{thmdifference}, that is, there is exactly one $l\geq 2$ such that $$\dim(\Gr_F^{n-p}H_l) = 1,$$ and the rest are 0. Moreover, by the same result, $\sum_{l\geq 2}{\dim(\Gr_F^{n-r}H_l)} = 1$, for $0\leq r\leq p-1$.

\begin{remark}\label{remarkfriedmanlaza} Friedman and Laza have studied related invariants of singularities in similar conditions in \cite{friedmanlaza22b}*{Theorem 6.11 and Corollary 6.14}.  
\end{remark}

\begin{definition}\label{(p,s-p)} Let $x\in D$ be an isolated singularity such that $\widetilde{\alpha_D}_x = p+1$, that is an isolated $p$-log-canonical that is not $p$-rational. Let $l$ be the degree such that $\dim(\Gr_F^{n-p}H_l) =1$. Then, we say that the singularity is of type $(p, n-l-p)$.
\end{definition}

\begin{remark} \begin{enumerate}[label=\roman*)] \item \definitionref{(p,s-p)} is analogous to the definition of isolated log-canonical singularities of type $(0,s)$ \cite{ishii85}*{Definition 4.1}, when $x\in D$ is an isolated singularity and $D$ is a hypersurface of a smooth variety.
\item Ishii defined these singularities more generally for normal isolated 1-Gorenstein log-canonical singularities. It is an open question how to generalize this definition for non-hypersurface singularities.
\item  The possible types are $(p,p), (p,p+1), \ldots (p,n-2-p)$. This is a consequence of the fact that the nilpotency order of the vanishing cohomology is bounded by Briançon-Skoda exponent \cite{scherk80}*{Main Theorem}. This nilpotency order gives a bound for the nilpotency order of $(\partial_t t)$ on $\gr_V^0B_f$, which in turn gives a bound for the order of $(t\partial_t)$ on $\gr_V^1B_f$. The Briançon-Skoda exponent is bounded by $n-2p -1$ (see for instance \cite{jksy22b}), which means that $n-l-p \geq p$.
\end{enumerate}\end{remark}

\begin{example}\label{examplessings} Suppose that $f\in\CC[x_1,\ldots ,x_n]$ is a polynomial with an isolated singularity at the origin, and a non-degenerate Newton boundary. Let $\Gamma_+(f) = \Gamma$ the Newton polyhedron of $f$, $\Gamma(f)$ the union of the compact faces of $\Gamma_+(f)$, and $\mathcal{F}$ the set of compact facets. For each $F\in\mathcal{F}$, there is a unique vector $B_F\in(\QQ_{\geq 0})^n$ such that $\langle A, B_F\rangle = 1$ for all $A\in F$. For every monomial $x^{\nu}$, we define $$\tilde{\rho}_F(x^{\nu}) = \langle \nu + \mathbf{1}, B_F\rangle ,$$ where $\mathbf{1} = (1,\ldots , 1)$, and for any $g\in\shO$, $g=\sum g_Ax^A$, $$\tilde{\rho}_F(g) = \min\{\tilde{\rho}_F(x^A)\ : \ g_A\neq 0\}.$$ Finally, we define $$\tilde{\rho}(g) = \min\{\tilde{\rho}_F(g) \ : \ F\in\mathcal{F}\}.$$ In this case, the minimal exponent is $\tilde{\rho}(1)$.\\

Suppose $\widetilde{\alpha}_f = p+1$, which implies that $\partial^p_t\delta \in V^1i_+\OX$. Using the description of the microlocal $V$-filtration (see \cite{saito94}*{Proposition 3.2}), we see that if $$r = \#\{F\in\mathcal{F} \; : \;  \tilde{\rho}_F(1) = \tilde{\rho}(1)\},$$ then $(t\partial_t)^{r+1}\partial_t^p\delta \in V^{>1}i_+\OX$, or equivalently, $$1\in I_p^{W_{r+1}}(D).$$ In general, $r+1$ is not the degree with $\Gr_F^{n-p}H_l\neq 0$.
\begin{enumerate}[label=\roman*), wide, labelwidth=!, labelindent=0pt]\item Weighted homogeneous singularities with $\widetilde{\alpha}_f=p+1$ are examples of singularities of type $(p,n-2-p)$ (see \remarkref{qhomog}). Isolated singularities with non-degenerate Newton boundary give examples for different degrees of $l$. For instance, $f = x^2 + y^2 + z^2 + u^2w^2 + u^4 + w^5\in\CC^5$ satisfies that $\widetilde{\alpha_f} = 2$, and $r=2$, using the notation above. We can also verify that $(t\partial_t)^{2}\partial_t\delta \notin V^{>1}i_+\OX$, since $w^5\partial_t^3\delta\in V^0\setminus V^{>0}$. Indeed, this follows from the fact that $w^5\notin J(f)$, where $J(f)$ is the Jacobian ideal, and \cite{jksy22a}*{Proposition 1.3}. Therefore, this singularity is of type $(1,1)$.\\
\item Let $\Delta_0$ be the compact face that contains $\frac{1}{p+1}\mathbf{1}$ in its relative interior, and let $s=\dim{\Delta_0}$. Assume also that the Newton polyhedron is simplicial. The number $r$ defined above satisfies that $s = n - r$. Let $l$ be the degree such that $\Gr_F^{n-p}H_l\neq 0$. Then $l\leq r+1 = n-s+1$, if $s>0$, and $l\leq n$ is $s=0$.\\

\item If $p=0$, the previous inequalities are equalities (without the simplicial assumption) by a result of Watanabe that says that the singularities are log-canonical of type $(0,s-1)$ if $s>0$, and $(0,0)$ if $s=0$, which is equivalent to the equalities \cite{watanabe86}*{Corollary 3.14}.
\end{enumerate} 
\end{example}

Using \propositionref{minexprk1} and the vanishing results, we obtain a bound on the number of these singularities in a hypersurface of $\PP^n$.

\begin{corollary}\label{countpoints} Let $D$ be a reduced hypersurface of $\PP^n$ of degree $d$ with at most isolated singularities. Assume that the pair $(\PP^n, D)$ is strictly $p$-log-canonical, that is, $\widetilde{\alpha_D} = p+1$. Let $Z$ be the union of the strictly $p$-log-canonical singular points of $D$ and $Z_2$ the union of those of type $(p,p), \ldots, (p, n-3-p)$. Then, $$\# Z_2 \leq \binom{(p+1)d - 1}{n},$$ and $$\# Z \leq \binom{(p+1)d }{n}.$$
\end{corollary}

\begin{proof} By \propositionref{minexprk1}, the scheme $Z$ is defined by the ideal $I_p^{W_l}(D)$. Therefore, the result follows from \corollaryref{indconditions}.

\end{proof}

\section{Restriction Theorem}

Let $(\Mmod, F)$ be a filtered right $\Dmod$-module underlying a mixed Hodge module $M$ on $X$. Let $H\subseteq X$ be a smooth hypersurface and $i:H\into X$ the inclusion. In this section, we change the notation of the $V$-filtration by $V_k = V^{-k}$, which is the notation used in \cite{mustatapopa18}. There exists a canonical morphism \begin{equation} \gr_0^V\Mmod \stackrel{\sigma}{\to} \gr_{-1}^V\Mmod\otimes_{\OX}\OX(H)\end{equation} satisfying $$\cohH^0i^!\Mmod\cong\ker(\sigma) \ \ \text{ and } \ \ \cohH^1i^!\Mmod\cong\coker(\sigma)$$
 with the filtrations induced by the filtrations on $\Mmod$ (see \cite{mustatapopa18}*{\textsection 2}). Moreover, on an open set $U\subseteq X$ where $H$ is given by a local equation $t$, this map corresponds to $$\Var = \cdot t: \gr^V_0\Mmod \to \gr^V_{-1}\Mmod$$ between the vanishing and nearby cycles along $H$.\\

In the proof of \cite{mustatapopa18}*{Theorem A} the authors defined for all $k$ a morphism \begin{equation}\label{maprestriction} F_k\cohH^1i^!\Mmod \to F_k\Mmod\otimes_{\OX}\shO_H(H).\end{equation} First, we define a morphism $$\eta: F_k\gr^V_{-1}\Mmod = \frac{F_kV_{-1}\Mmod}{F_kV_{<-1}\Mmod} \to F_k\Mmod\otimes_{\OX}\shO_H $$ such that for $u\in F_kV_{-1}\Mmod$, $\eta(u)$ is the class of $u$ in $F_k\Mmod\otimes_{\OX}\shO_H$. This map is well defined, as on an open set $U$ where $H$ is defined by an equation $t$, the $V$-filtration satisfies $$(F_kV_{\alpha}\Mmod)\cdot t = F_kV_{\alpha-1}\Mmod \ \text{ for } \alpha<0,$$ and $F_k\Mmod\cdot t$ maps to 0 in $F_k\Mmod\otimes_{\OX}\shO_H$. The map $\eta$ induces a map on $F_k\cohH^1i^!\Mmod$. Indeed, since locally $\sigma$ is right multiplication by $t$, the image of $\sigma$ is mapped to 0 by $\eta\otimes\OX(H)$.\\

\begin{proof}[Proof of \theoremref{thmrestriction}] Let $\Mmod = W_{n+l}\omega_X(*D)$. For every $k$ we have the canonical morphism (\ref{maprestriction}): $$F_k\cohH^1i^!\Mmod \to F_k\Mmod\otimes_{\OX}\shO_H(H).$$ Note that the sheaf $$F_{k-n}\Mmod\otimes_{\OX}\shO_H(H) = I_k^{W_l}(D)\otimes\omega_X((k+1)D) \otimes\shO_H(H) \cong I_k^{W_l}(D)\otimes\omega_H((k+1)D_H).$$\\

Consider the short exact sequence $$0\to \Mmod\to \omega_X(*D)\to\mathcal{C} \to 0.$$ Applying the functor $i^!$ and taking cohomology we obtain an exact sequence $$0\to \cohH^0i^1\mathcal{C}\to\cohH^1i^!\Mmod\to\cohH^1i^!\omega_X(*D)\to\cohH^1i^!\mathcal{C}\to 0,$$ as $\cohH^0i^!\omega_X(*D)=0$. As $\gr^W_i\mathcal{C} = 0$ for $i<n+l+1$, $$\gr^W_i\cohH^0i^!\mathcal{C}= 0 \ \ \text{ for } i<n+l+1,$$ and $$\gr^W_i\cohH^1i^!\mathcal{C}= 0 \ \ \text{ for } i<n+l+2$$ by \cite{saito90}*{Proposition 2.26}. Therefore, we obtain a short exact sequence $$ 0\to W_{n+l+1}\cohH^0i^1\mathcal{C}\to W_{n+l+1}\cohH^1i^!\Mmod\to W_{n+l+1}\cohH^1i^!\omega_X(*D)\to 0.$$ Note that as $$\Ext^1(W_{n+l+1}\cohH^1i^!\omega_X(*D), W_{n+l+1}\cohH^0i^1\mathcal{C}) =0$$ (see \cite{schnellsurvey}*{\textsection 23}), there is a split map \begin{equation}\label{splitmap} W_{n+l+1}\cohH^1i^!\omega_X(*D) \to W_{n+l+1}\cohH^1i^!\Mmod.\end{equation} The source of this maps admits the following interpretation: $$W_{n+l+1}\cohH^1i^!\omega_X(*D)\cong W_{n-1+l}\omega_H(*D_H).$$ Indeed, $$\cohH^1i^!\omega_X(*D) \cong \omega_H(*D_H)(-1)$$ \cite{mustatapopa18}*{Proof of Theorem A}.\\

Taking the corresponding piece of the Hodge filtration in (\ref{splitmap}) and composing it with (\ref{maprestriction}), we obtain a morphism  $$F_kW_{n+l+1}\cohH^1i^!\omega_X(*D) \to F_k\Mmod\otimes_{\OX}\shO_H(H).$$ Using the morphism above and switching $k$ to $k-n$, we obtain a map $$ F_{k-n+1}W_{n-1+l}\omega_H(*D_H) = I_k^{W_l}(D_H)\otimes\omega_H((k+1)D_H) \to I_k^{W_l}(D)\otimes\omega_H((k+1)D_H),$$ and hence $$ I_k^{W_l}(D_H) \to I_k^{W_l}(D)\otimes\shO_H.$$ Composing this map with $I_k^{W_l}(D)\otimes\shO_H \to I_k^{W_l}(D)\cdot\shO_H,$ we obtain a morphism \begin{equation}\label{equationinclusionideals} I_k^{W_l}(D_H) \to I_k^{W_l}(D)\cdot\shO_H.\end{equation} By construction, this map is compatible with restriction to open sets. Let $V=H\setminus D_H$ be the complement. When restricted to $V$, this map is the identity on $\shO_V$, and therefore it is an inclusion.\\

For the last statement, we note that a general $H$ is in particular non-characteristic with respect to $\omega_X(*D)$. By the description of the $V$-filtration in this case \cite{saito88}*{Lemma 3.5.6}, the map $\sigma$ is the zero map, and therefore (\ref{maprestriction}) is a surjection. Moreover, in this case $$\cohH^1i^!\Mmod = \cohH^1i^!W_{n+l}\omega_X(*D) \cong W_{n+l+1}\cohH^1i^!\omega_X(*D)$$ where the first equality is the definition of $\Mmod$ and the isomorphism is a result of Saito \cite{saito90}*{Lemma 2.25}. Hence, in this case (\ref{equationinclusionideals}) is an isomorphism.
\end{proof}

\begin{remark} A similar result can be obtained when $H$ is an intersection of several general hyperplane sections. For more details, see \cite{olano22a}*{Remark 12}.
\end{remark}

\bibliography{../../bib}

\end{document}

%% file: macros.tex
\usepackage{amsmath,amsfonts,amsthm,mathrsfs}
\usepackage{amssymb, stmaryrd}

\usepackage[unicode,bookmarks]{hyperref}
\usepackage[usenames,dvipsnames]{xcolor}
\hypersetup{colorlinks=true,citecolor=NavyBlue,linkcolor=Brown,urlcolor=Orange}

\usepackage[alphabetic,initials]{amsrefs}

\usepackage{enumitem}

\newenvironment{aenumerate}{%
	\begin{enumerate}[label=(\alph{*}), ref=(\alph{*})]
}{%
	\end{enumerate}%
}

\usepackage{chngcntr}

\ifdraft
\usepackage[notcite,notref,color]{showkeys}

\definecolor{labelkey}{gray}{0.5}
\fi

\usepackage{tikz}

\usepackage{tikz-cd}
\usetikzlibrary{matrix,arrows}
\newlength{\myarrowsize} 

\pgfarrowsdeclare{cmto}{cmto}{
	\pgfsetdash{}{0pt} 
	\pgfsetbeveljoin 
	\pgfsetroundcap 
	\setlength{\myarrowsize}{0.6pt}
	\addtolength{\myarrowsize}{.5\pgflinewidth}
	\pgfarrowsleftextend{-4\myarrowsize-.5\pgflinewidth} 
	\pgfarrowsrightextend{.8\pgflinewidth}
}{
	\setlength{\myarrowsize}{0.6pt} 
  	\addtolength{\myarrowsize}{.5\pgflinewidth}  
	\pgfsetlinewidth{0.5\pgflinewidth}
	\pgfsetroundjoin
	\pgfpathmoveto{\pgfpoint{1.5\pgflinewidth}{0}}
	\pgfpatharc{-109}{-170}{4\myarrowsize}
	\pgfpatharc{10}{189}{0.58\pgflinewidth and 0.2\pgflinewidth}
	\pgfpatharc{-170}{-115}{4\myarrowsize+\pgflinewidth}
	\pgfpathclose
	\pgfusepathqfillstroke
	\pgfpathmoveto{\pgfpoint{1.5\pgflinewidth}{0}}
	\pgfpatharc{109}{170}{4\myarrowsize}
	\pgfpatharc{-10}{-189}{0.58\pgflinewidth and 0.2\pgflinewidth}
	\pgfpatharc{170}{115}{4\myarrowsize+\pgflinewidth}
	\pgfpathclose
	\pgfusepathqfillstroke
	\pgfsetlinewidth{2\pgflinewidth}
}

\pgfarrowsdeclare{cmonto}{cmonto}{
	\pgfsetdash{}{0pt} 
	\pgfsetbeveljoin 
	\pgfsetroundcap 
	\setlength{\myarrowsize}{0.6pt}
	\addtolength{\myarrowsize}{.5\pgflinewidth}
	\pgfarrowsleftextend{-4\myarrowsize-.5\pgflinewidth} 
	\pgfarrowsrightextend{.8\pgflinewidth}
}{
	\setlength{\myarrowsize}{0.6pt} 
  	\addtolength{\myarrowsize}{.5\pgflinewidth}  
	\pgfsetlinewidth{0.5\pgflinewidth}
	\pgfsetroundjoin
	\pgfpathmoveto{\pgfpoint{1.5\pgflinewidth}{0}}
	\pgfpatharc{-109}{-170}{4\myarrowsize}
	\pgfpatharc{10}{189}{0.58\pgflinewidth and 0.2\pgflinewidth}
	\pgfpatharc{-170}{-115}{4\myarrowsize+\pgflinewidth}
	\pgfpathclose
	\pgfusepathqfillstroke
	\pgfpathmoveto{\pgfpoint{1.5\pgflinewidth}{0}}
	\pgfpatharc{109}{170}{4\myarrowsize}
	\pgfpatharc{-10}{-189}{0.58\pgflinewidth and 0.2\pgflinewidth}
	\pgfpatharc{170}{115}{4\myarrowsize+\pgflinewidth}
	\pgfpathclose
	\pgfusepathqfillstroke
	\pgfpathmoveto{\pgfpoint{1.5\pgflinewidth-0.3em}{0}}
	\pgfpatharc{-109}{-170}{4\myarrowsize}
	\pgfpatharc{10}{189}{0.58\pgflinewidth and 0.2\pgflinewidth}
	\pgfpatharc{-170}{-115}{4\myarrowsize+\pgflinewidth}
	\pgfpathclose
	\pgfusepathqfillstroke
	\pgfpathmoveto{\pgfpoint{1.5\pgflinewidth-0.3em}{0}}
	\pgfpatharc{109}{170}{4\myarrowsize}
	\pgfpatharc{-10}{-189}{0.58\pgflinewidth and 0.2\pgflinewidth}
	\pgfpatharc{170}{115}{4\myarrowsize+\pgflinewidth}
	\pgfpathclose
	\pgfusepathqfillstroke
	\pgfsetlinewidth{2\pgflinewidth}
}

\pgfarrowsdeclare{cmhook}{cmhook}{
	\pgfsetdash{}{0pt} 
	\pgfsetbeveljoin 
	\pgfsetroundcap 
	\setlength{\myarrowsize}{0.6pt}
	\addtolength{\myarrowsize}{.5\pgflinewidth}
	\pgfarrowsleftextend{-4\myarrowsize-.5\pgflinewidth} 
	\pgfarrowsrightextend{.8\pgflinewidth}
}{
	\setlength{\myarrowsize}{0.6pt} 
  	\addtolength{\myarrowsize}{.5\pgflinewidth}  
 	\pgfsetdash{}{0pt}
	\pgfsetroundcap
	\pgfpathmoveto{\pgfqpoint{0pt}{-4.667\pgflinewidth}}
	\pgfpathcurveto
    {\pgfqpoint{4\pgflinewidth}{-4.667\pgflinewidth}}
    {\pgfqpoint{4\pgflinewidth}{0pt}}
    {\pgfpointorigin}
	\pgfusepathqstroke
}

\newenvironment{diagram*}[2]{%
\[%
\begin{tikzpicture}[>=cmto,baseline=(current bounding box.center),%
	to/.style={->,font=\scriptsize,cap=round},%
	into/.style={cmhook->,font=\scriptsize,cap=round},%
	onto/.style={-cmonto,font=\scriptsize,cap=round},%
	math/.style={matrix of math nodes, row sep=#2, column sep=#1,%
		text height=1.5ex, text depth=0.25ex}]%
}{%
\end{tikzpicture}%
\]% 
\ignorespacesafterend%
}

\newcommand{\Dmod}{\mathscr{D}}
\newcommand{\Mmod}{\mathcal{M}}

\newcommand{\derL}{\mathbf{L}}
\newcommand{\cohH}{\mathcal{H}}
\newcommand{\Ltensor}{\overset{\derL}{\tensor}}

\newcommand{\tensor}{\otimes}

\newcommand{\shTor}{\mathcal{T}\hspace{-2.5pt}\mathit{or}}

\newcommand{\ZZ}{\mathbb{Z}}
\newcommand{\QQ}{\mathbb{Q}}

\newcommand{\CC}{\mathbb{C}}
\newcommand{\HH}{\mathbb{H}}
\newcommand{\PP}{\mathbb{P}}

\DeclareMathOperator{\im}{im}

\DeclareMathOperator{\coker}{coker}

\DeclareMathOperator{\Spec}{Spec}

\DeclareMathOperator{\Supp}{Supp}

\DeclareMathOperator{\gr}{gr}
\DeclareMathOperator{\DR}{DR}

\DeclareMathOperator{\Ext}{Ext}

\DeclareMathOperator{\Var}{Var}

\newcommand{\shf}[1]{\mathscr{#1}}
\newcommand{\OX}{\shf{O}_X}

\newcommand{\restr}[1]{\big\vert_{#1}}

\def\overbar#1#2#3{{%
	\setbox0=\hbox{\displaystyle{#1}}%
	\dimen0=\wd0
	\advance\dimen0 by -#2 
	\vbox {\nointerlineskip \moveright #3 \vbox{\hrule height 0.3pt width \dimen0}%
		\nointerlineskip \vskip 1.5pt \box0}%
}}

\newcommand{\into}{\hookrightarrow}
\newcommand{\onto}{\to\hspace{-0.7em}\to}

\newcommand{\shO}{\shf{O}}

\makeatletter
\let\@@seccntformat\@seccntformat
\renewcommand*{\@seccntformat}[1]{%
  \expandafter\ifx\csname @seccntformat@#1\endcsname\relax
    \expandafter\@@seccntformat
  \else
    \expandafter
      \csname @seccntformat@#1\expandafter\endcsname
  \fi
    {#1}%
}
\newcommand*{\@seccntformat@subsection}[1]{%
  \textbf{\csname the#1\endcsname.}
}
\makeatother

\makeatletter
\let\@paragraph\paragraph
\renewcommand*{\paragraph}[1]{%
	\vspace{0.3\baselineskip}%
	\@paragraph{\textit{#1}}%
}
\makeatother

\counterwithin{equation}{subsection}
\counterwithout{subsection}{section}
\counterwithin{figure}{subsection}

\newtheorem*{theorem*}{Theorem}
\newtheorem{lemma}[equation]{Lemma}
\newtheorem*{lemma*}{Lemma}
\newtheorem{corollary}[equation]{Corollary}
\newtheorem{proposition}[equation]{Proposition}
\newtheorem*{proposition*}{Proposition}

\theoremstyle{definition}
\newtheorem{definition}[equation]{Definition}
\newtheorem*{definition*}{Definition}
\newtheorem{remark}[equation]{Remark}

\newtheorem{example}[equation]{Example}
\newtheorem*{example*}{Example}
\newtheorem*{problem*}{Problem}

\theoremstyle{plain}

\newcommand{\theoremref}[1]{\hyperref[#1]{Theorem~\ref*{#1}}}
\newcommand{\lemmaref}[1]{\hyperref[#1]{Lemma~\ref*{#1}}}
\newcommand{\definitionref}[1]{\hyperref[#1]{Definition~\ref*{#1}}}
\newcommand{\propositionref}[1]{\hyperref[#1]{Proposition~\ref*{#1}}}
\newcommand{\conjectureref}[1]{\hyperref[#1]{Conjecture~\ref*{#1}}}
\newcommand{\corollaryref}[1]{\hyperref[#1]{Corollary~\ref*{#1}}}
\newcommand{\exampleref}[1]{\hyperref[#1]{Example~\ref*{#1}}}

\makeatletter
\let\old@caption\caption
\renewcommand*{\caption}[1]{%
	\setcounter{figure}{\value{equation}}%
	\stepcounter{equation}%
	\old@caption{#1}\relax%
}
\makeatother

\newcounter{intro}

\newtheorem{intro-conjecture}[intro]{Conjecture}
\newtheorem{intro-corollary}[intro]{Corollary}
\newtheorem{intro-theorem}[intro]{Theorem}

\newcommand{\OY}{\shO_Y}

\newcommand{\parref}[1]{\hyperref[#1]{\S\ref*{#1}}}

\makeatletter
\newcommand*\if@single[3]{%
  \setbox0\hbox{${\mathaccent"0362{#1}}^H$}%
  \setbox2\hbox{${\mathaccent"0362{\kern0pt#1}}^H$}%
  \ifdim\ht0=\ht2 #3\else #2\fi
  }
\newcommand*\rel@kern[1]{\kern#1\dimexpr\macc@kerna}
\newcommand*\widebar[1]{\@ifnextchar^{{\wide@bar{#1}{0}}}{\wide@bar{#1}{1}}}
\newcommand*\wide@bar[2]{\if@single{#1}{\wide@bar@{#1}{#2}{1}}{\wide@bar@{#1}{#2}{2}}}
\newcommand*\wide@bar@[3]{%
  \begingroup
  \def\mathaccent##1##2{%
    \if#32 \let\macc@nucleus\first@char \fi
    \setbox\z@\hbox{$\macc@style{\macc@nucleus}_{}$}%
    \setbox\tw@\hbox{$\macc@style{\macc@nucleus}{}_{}$}%
    \dimen@\wd\tw@
    \advance\dimen@-\wd\z@
    \divide\dimen@ 3
    \@tempdima\wd\tw@
    \advance\@tempdima-\scriptspace
    \divide\@tempdima 10
    \advance\dimen@-\@tempdima
    \ifdim\dimen@>\z@ \dimen@0pt\fi
    \rel@kern{0.6}\kern-\dimen@
    \if#31
      \overline{\rel@kern{-0.6}\kern\dimen@\macc@nucleus\rel@kern{0.4}\kern\dimen@}%
      \advance\dimen@0.4\dimexpr\macc@kerna
      \let\final@kern#2%
      \ifdim\dimen@<\z@ \let\final@kern1\fi
      \if\final@kern1 \kern-\dimen@\fi
    \else
      \overline{\rel@kern{-0.6}\kern\dimen@#1}%
    \fi
  }%
  \macc@depth\@ne
  \let\math@bgroup\@empty \let\math@egroup\macc@set@skewchar
  \mathsurround\z@ \frozen@everymath{\mathgroup\macc@group\relax}%
  \macc@set@skewchar\relax
  \let\mathaccentV\macc@nested@a
  \if#31
    \macc@nested@a\relax111{#1}%
  \else
    \def\gobble@till@marker##1\endmarker{}%
    \futurelet\first@char\gobble@till@marker#1\endmarker
    \ifcat\noexpand\first@char A\else
      \def\first@char{}%
    \fi
    \macc@nested@a\relax111{\first@char}%
  \fi
  \endgroup
}
\makeatother